\documentclass[
paper=a4,%
pagesize,%
numbers=noendperiod,%
captions=nooneline,%
abstracton%
]{scrartcl}

\newcommand*{\mytitle}{Densities\ of\ nested\ Archimedean\ copulas}
\newcommand*{\myauthorone}{Marius\ Hofert}
\newcommand*{\mycontactone}{RiskLab,\ Department\ of\ Mathematics,\ ETH\ Zurich,\ 8092\ Zurich,\ Switzerland,\ \href{mailto:marius.hofert@math.ethz.ch}{\nolinkurl{marius.hofert@math.ethz.ch}}}
\newcommand*{\mycomment}{The\ author\ (Willis\ Research\ Fellow)\ thanks\ Willis\ Re\ for\ financial\ support\ while\ this\ work\ was\ being\ completed.}
\newcommand*{\myauthortwo}{David\ Pham}
\newcommand*{\mycontacttwo}{Department\ of\ Mathematics\ and\ Statistics,\ University\ of\ Montreal,\ H3T\ 1J4\ Montreal,\ Quebec,\ Canada,\ \href{mailto:david-olivier.pham@umontreal.ca}{\nolinkurl{david-olivier.pham@umontreal.ca}}}
\newcommand*{\mysubject}{Article}

\usepackage[T1]{fontenc}%
\usepackage{lmodern}%
\usepackage[american]{babel}%
\usepackage{microtype}%

\usepackage[nouppercase]{scrpage2}%

\usepackage{eso-pic}%
\usepackage{rotating}%
\usepackage{amsmath}%
\usepackage{mathtools}%
\usepackage{amssymb}%
\usepackage{amsthm}%
\usepackage{bm}%
\usepackage{bbm}%
\usepackage{enumitem}%
\usepackage{graphicx}%
\usepackage{grffile}%
\usepackage{tikz}%
\usepackage{wrapfig}%
\usepackage{tabularx}%
\usepackage{dcolumn}%
\usepackage{booktabs}%
\usepackage{multirow}%
\usepackage{caption}%
\usepackage{listings}%
\usepackage{vruler}%
\usepackage{csquotes}%
\usepackage[
backend=bibtex,%
style=authoryear,%
dashed=false,%
uniquename=init,%
firstinits=true,%
hyperref=true,%
maxcitenames=2,%
maxbibnames=6,%
date=iso8601,%
urldate=iso8601%
]{biblatex}%
\usepackage[
hypertexnames=false,%
setpagesize=false,%
pdfborder={0 0 0},%
pdfstartview=Fit,%
bookmarksopen=true,%
bookmarksnumbered=true%
]{hyperref}%
\hypersetup{
  pdfauthor=\myauthorone,%
  pdftitle=\mytitle,%
  pdfsubject=\mysubject%
}
\setkomafont{pageheadfoot}{\normalfont\normalcolor\sffamily}%
\automark{section}%
\setkomafont{pagenumber}{\normalfont\normalcolor\sffamily}%
\setkomafont{captionlabel}{\normalfont\normalcolor\sffamily\bfseries}%

\makeatletter
\newcommand\myisodate{\number\year-\ifcase\month\or 01\or 02\or 03\or 04\or 05\or 06\or 07\or 08\or 09\or 10\or 11\or 12\fi-\ifcase\day\or 01\or 02\or 03\or 04\or 05\or 06\or 07\or 08\or 09\or 10\or 11\or 12\or 13\or 14\or 15\or 16\or 17\or 18\or 19\or 20\or 21\or 22\or 23\or 24\or 25\or 26\or 27\or 28\or 29\or 30\or 31\fi}%
\makeatother
\newcolumntype{d}[2]{D{.}{.}{#1.#2}}%
\newcommand*{\abstractnoindent}{}%
\let\abstractnoindent\abstract
\renewcommand*{\abstract}{\let\quotation\quote\let\endquotation\endquote
  \abstractnoindent}
\deffootnote[1em]{1em}{1em}{\textsuperscript{\thefootnotemark}}%

\DeclareNameAlias{sortname}{last-first}%
\DefineBibliographyExtras{american}{\DeclareQuotePunctuation{}}%
\renewbibmacro*{volume+number+eid}{%
  \setunit*{\addcomma\space}%
  \printfield{volume}%
  \printfield{number}}
\DeclareFieldFormat*{number}{(#1)}
\DeclareFieldFormat*{title}{#1}%
\DeclareFieldFormat{doi}{%
  \ifhyperref
    {\href{http://dx.doi.org/#1}{\nolinkurl{doi:#1}}}%
    {\nolinkurl{doi:#1}}}%
\renewbibmacro*{in:}{}%
\DeclareFieldFormat{isbn}{ISBN #1}%
\DeclareFieldFormat{pages}{#1}%
\DeclareFieldFormat{url}{\url{#1}}%
\DeclareFieldFormat{urldate}{\mkbibparens{#1}}%
\bibliography{./mybib}%
\renewcommand*{\cite}[2][]{\textcite[#1]{#2}}

\newtheoremstyle{mythmstyle}%
{0.5em}%
{0.5em}%
{}%
{}%
{\sffamily\bfseries}%
{}%
{\newline}%
{\thmname{#1}\ \thmnumber{#2}\ \thmnote{(#3)}}%
\newcommand*{\myskip}{~\vspace{-1.2em}}%
\theoremstyle{mythmstyle}
\newtheorem{definition}{Definition}[section]%
\newtheorem{proposition}[definition]{Proposition}
\newtheorem{lemma}[definition]{Lemma}
\newtheorem{theorem}[definition]{Theorem}

\newtheorem{remark}[definition]{Remark}

\renewcommand*\proofname{Proof}
\makeatletter%
\renewenvironment{proof}[1][\proofname]{\par
  \pushQED{\qed}%
  \normalfont\topsep2\p@\@plus2\p@\relax
  \trivlist
\item[\hskip\labelsep
  \sffamily\bfseries #1]\mbox{}\hfill\\*\ignorespaces
}{%
  \popQED\endtrivlist\@endpefalse
}
\makeatother

\renewcommand*{\cite}[2][]{\textcite[#1]{#2}}
\newlength{\mytopsep}
\newlength{\myitemsep}
\newcommand{\T}{^{\top}}

\newcommand*{\IN}{\mathbb{N}}

\newcommand*{\IR}{\mathbb{R}}

\newcommand*{\IE}{\mathbb{E}}

\newcommand*{\sign}{\operatorname*{sign}}

\newcommand*{\Li}{\operatorname*{Li}}

\newcommand*{\cpsi}{{\mathring{\psi}}}
\newcommand*{\psii}{{\psi^{-1}}}
\newcommand*{\psis}[2]{{\psi_{#1}^{#2}}}
\newcommand*{\psiis}[1]{{\psi_{#1}^{-1}}}

\newcommand*{\LS}{\mathcal{LS}}
\newcommand*{\LSi}{\LS^{-1}}

\newcommand*{\var}[3]{{{#1}^{#2}_{#3}}}

\begin{document}
\thispagestyle{plain}
\begin{center}
  \sffamily
  {\bfseries\LARGE\mytitle\par}
  \bigskip
  {\Large\myauthorone\footnote{\mycontactone. \mycomment},\ \myauthortwo\footnote{\mycontacttwo}\par
    \bigskip
    \myisodate\par}
\end{center}
\par\bigskip
\begin{abstract}
  Nested Archimedean copulas recently gained interest since they generalize the
  well-known class of Archimedean copulas to allow for partial
  asymmetry. Sampling algorithms and strategies have been well investigated for
  nested Archimedean copulas. However, for likelihood based inference it is
  important to have the density. The present work fills this gap. A general
  formula for the derivatives of the nodes and inner generators appearing in
  nested Archimedean copulas is developed. This leads to a tractable formula for
  the density of nested Archimedean copulas in arbitrary dimensions if the
  number of nesting levels is not too large. Various examples including famous
  Archimedean families and transformations of such are given. Furthermore, a
  numerically efficient way to evaluate the log-density is presented.
\end{abstract}
\minisec{Keywords}
Nested Archimedean copulas, generator derivatives, likelihood-based inference.
\minisec{MSC2010}
62H99, 65C60, 62H12, 62F10.%
\section{Introduction}\label{sec.intro}
There has recently been interest in multivariate hierarchical models, that is,
models that are able to capture different dependencies between and within
different groups of random variables. One such class of models is based on
nested Archimedean copulas. A \emph{partially nested Archimedean copula} $C$
with two nesting \emph{levels} and $d_0$ \emph{child copulas} (or \emph{sectors} or \emph{groups}), is given by
\begin{align}
  C(\bm{u})=C_0(C_1(\bm{u}_1),\dots,C_{d_0}(\bm{u}_{d_0})),\ \bm{u}=(\bm{u}_1,\dots,\bm{u}_{d_0})\T,\label{nac}
\end{align}
where $d_0$ denotes the dimension of $C_0$ and each copula $C_s$,
$s\in\{0,\dots,d_0\}$, is Archimedean with a completely monotone \emph{generator} $\psi_s$, that is,
\begin{align}
  C_s(\bm{u}_s)=\psi_s(\psiis{s}(u_{s1})+\dots+\psiis{s}(u_{sd_s}))=\psi_s(t_s(\bm{u}_s)),\label{eq:AC}
\end{align}
where
\begin{align*}
  t_s(\bm{u}_s)=\sum_{j=1}^{d_s}\psiis{s}(u_{sj})
\end{align*}
and $\psi_s:[0,\infty]\to[0,1]$ is continuous, $\psi_s(0)=1$,
$\psi_s(\infty)=\lim_{t\to\infty}\psi_s(t)=0$, and $(-1)^k\psi_s^{(k)}(t)\ge0$
for all $k\in\IN_{0}$, $t\in(0,\infty)$. The set of all completely monotone
Archimedean generators is denoted by $\Psi_\infty$ in what follows. The copula $C_0$ is referred to as \emph{root copula}. Model (\ref{nac})
provides an intuitive hierarchical structure, since, for example, if $\bm{U}\sim
C$ the pair $(U_{sj},U_{sk})\T$ ($j\neq k$) has joint copula $C_s$ whereas the
pair $(U_{rj},U_{sk})\T$ ($r\neq s$) follows the root copula $C_0$. One can
therefore directly say what the bivariate margins are and theoretical results
about measures of association, for example, directly apply. Furthermore, such a
construction provides an explicit form for the copula itself, which is important,
for example, in likelihood-based inference for censored data. More
complicated nesting structures can be constructed; see Section \ref{sec.ext}. In general, a \emph{nested Archimedean copula} is an Archimedean copula with arguments possibly replaced by other nested Archimedean copulas.

For statistical applications it is desirable to be able to evaluate the density
of a multivariate model, for example, for parameter estimation or when
conditional copulas are required (for example, for goodness-of-fit testing via Rosenblatt's
transform; see \cite{genestremillardbeaudoin2009}). For Archimedean copulas, the density (if it exists) is theoretically trivial to write down; for (\ref{eq:AC}), one obtains
\begin{align*}
  c_s(\bm{u}_s)=\psis{s}{(d)}(t_s(\bm{u}_s))\prod_{j=1}^{d_s}(\psiis{s})^\prime(u_{sj}),\ \bm{u}_s\in(0,1)^{d_s}.
\end{align*}
However, the appearing generator derivatives $\psis{s}{(d)}$ are non-trivial to
access theoretically and, even more, computationally. This issue has recently
been solved for several well-known Archimedean copulas and transformations of
such; see \cite{hofertmaechlermcneil2012b} and
\cite{hofertmaechlermcneil2012a}. Our goal is to extend these results to the corresponding nested Archimedean copulas. Note that this is more challenging because differentiating (\ref{nac}) is more complicated due to the inner derivatives that appear when applying the Chain Rule; in contrast to Archimedean copulas, these inner derivatives depend on variables with respect to which one has to differentiate again. Already in low dimensions the corresponding formulas for the density $c$ become challenging to write down and, even more, to evaluate in a numerically stable way.

After a brief introduction and overview about nested Archimedean copulas in
Section \ref{sec.nac}, we tackle the problem of computing the density of
(\ref{nac}) by first deriving a convenient form for the copula. This will allow
us to compute the density; see Section \ref{sec.dnac}. All necessary details for
several well-known Archimedean families are provided in Section
\ref{sec.ex}. Section \ref{sec.num} addresses numerical evaluation of the
log-density. Section \ref{sec.ext} presents the density for three-level nested
Archimedean copulas and extensions to higher nesting levels are briefly
addressed. %
For the reader's convenience, proofs are deferred to the appendix.

\section{Nested Archimedean copulas}\label{sec.nac}
The class of nested Archimedean copulas was first considered in \cite[p.\
87]{joe1997} in the three- and four-dimensional case and later by
\cite{mcneil2008} in the general $d$-dimensional case. \cite{mcneil2008} and
\cite{hofert2011b} derive an explicit stochastic representation for nested
Archimedean copulas which allows for a fast sampling algorithm of nested
Archimedean copulas similar to the Marshall--Olkin algorithm for Archimedean
copulas; see \cite{marshallolkin1988} for the latter. \cite{hofert2011a} provides efficient sampling strategies for the most important ingredients to this algorithm, the random variables responsible for introducing hierarchical dependencies. An implementation for several well-known Archimedean families (and transformations of such) is provided by the \textsf{R} package \texttt{copula}; see \cite{hofertmaechler2011}.

Although nesting is possible in more complicated ways (see Section
\ref{sec.ext}), in the following we focus on nested Archimedean copulas of Type
(\ref{nac}) (with some child copulas possibly shrunk to single arguments of
$C_0$). By Bernstein's Theorem, each $\psi\in\Psi_\infty$ is the Laplace--Stieltjes transform of a distribution function $F$ on
$[0,\infty)$ with $F(0)=0$. A sufficient condition under which (\ref{nac}) is
indeed a proper copula is then that the \emph{nodes}
\begin{align*}
  \cpsi_{0s}=\psiis{0}\circ\psi_s,\ s\in\{1,\dots,d_0\},
\end{align*}
have completely monotone first order derivatives; see \cite{mcneil2008}. Note that this \emph{sufficient nesting condition} is indeed only sufficient but not necessary. For example, if $\psi_0(t)=-\log(1-(1-e^{-\theta_0})\exp(-t))/\theta_0$ denotes the generator of a Frank copula and $\psi_1(t)=(1+t)^{-1/\theta_1}$ the generator of a Clayton copula, then $C(\bm{u})=C_0(u_1,C_1(u_2,u_3))$ is a valid (nested Archimedean) copula for all $\theta_0,\theta_1$ such that $\theta_0/(1-e^{-\theta_0})-1\le\theta_1$ although $\cpsi_{01}$ is not completely monotone for any parameters $\theta_0,\theta_1$.

Among the most widely used parametric Archimedean families are those of
Ali--Mikhail--Haq, Clayton, Frank, Gumbel, and Joe; see
\cite{hofertmaechlermcneil2012a} for the corresponding generators, their
derivatives, Laplace-Stieltjes inverses, and properties of the copula
families. These one-parameter families can easily be extended to allow for more
parameters, for example, via outer power transformations. For more details on
this and other aspects of nested Archimedean copulas we refer to
\cite{hofert2010c} and the references therein.

\section{Inner generator derivatives and densities for two-level nested Archimedean copulas}\label{sec.dnac}
\subsection{The basic idea}\label{sec.basic}
Let $C$ be a $d$-dimensional nested Archimedean copula of Type (\ref{nac}) (with
some child copulas possibly shrunk to single arguments of $C_0$) and assume the
sufficient nesting condition to hold; for the Ali--Mikhail--Haq, Clayton, Frank,
Gumbel, and Joe families, this is fulfilled as long as all generators belong to the same family and $\theta_0\le\theta_s$, $s\in\{1,\dots,d_0\}$. This condition implies that each copula $C_s$, $s\in\{1,\dots,d_0\}$, is more concordant than $C_0$.

One of the main ingredients we need in the following is the function
\begin{align}
  \psi_{0s}(t;v)=\exp(-v\cpsi_{0s}(t))\label{psi0s}
\end{align}
which we refer to as \emph{inner generator}. It is a proper generator in $t$ for
each $v>0$ as a composition of the completely monotone function
$\exp(-v\,\cdot)$ with $\cpsi_{0s}$ which has completely monotone derivative. With $F_0=\LSi[\psi_0]$, we obtain
\begin{align}
  C(\bm{u})&=C_0(C_1(\bm{u}_1),\dots,C_{d_0}(\bm{u}_{d_0}))=\int_0^\infty\exp\biggl(-v_0\sum_{s=1}^{d_0}\cpsi_{0s}(t_s(\bm{u}_s))\biggr)\,dF_0(v_0)\notag\\
&=\int_0^\infty\prod_{s=1}^{d_0}\psi_{0s}(t_s(\bm{u}_s);v_0)\,dF_0(v_0)\label{nac.repr}
\end{align}
By our assumption of having completely monotone generators, the density $c$ of
$C$ exists and is given by
\begin{align*}
  c(\bm{u})=\frac{\partial^d}{\partial u_{d_0d_{d_0}}\cdots\partial u_{11}}C(\bm{u}).
\end{align*}
Instead of differentiating (\ref{nac}) directly, the idea is now to use
Representation (\ref{nac.repr}). By differentiating under the integral sign, the
density $c$ allows for the representation
\begin{align}
c(\bm{u})&=\int_0^\infty\prod_{s=1}^{d_0}\psis{0s}{(d_s)}(t_s(\bm{u}_s);v_0)\,dF_0(v_0)\cdot\prod_{s=1}^{d_0}\prod_{j=1}^{d_s}(\psiis{s})^\prime(u_{sj})\notag\\
&=\IE\biggl[\,\prod_{s=1}^{d_0}\psis{0s}{(d_s)}(t_s(\bm{u}_s);V_0)\biggr]\cdot\prod_{s=1}^{d_0}\prod_{j=1}^{d_s}(\psiis{s})^\prime(u_{sj}).\label{dnac}
\end{align}
For the cost of one integral (which will be computed
explicitly below!), one can therefore easily compute the density $c$ (theoretically) as a
$F_0$-mixture. This is especially advantageous in large dimensions
as the complexity of the problem does not (again, theoretically) depend on the
sizes of the child copulas too much, rather on the number of children.

From Equation (\ref{dnac}), we identify the following key challenges:
\begin{enumerate}[label=\sffamily\bfseries Challenge \arabic*\ ,leftmargin=*,align=left,topsep=\mytopsep,itemsep=\myitemsep]
\item Find the derivatives of the inner generators $\psi_{0s}(t;v_0)$;
\item Compute their product;
\item Integrate it with respect to the mixture distribution function $F_0=\LSi[\psi_0]$.
\end{enumerate}
All three challenges will be solved in Section \ref{sec.main} with the help of
the tools presented in the following section.

\subsection{The tools needed: Fa{\`a} di Bruno's formula and Bell polynomials}
One formula which proves to be useful here, is the expression of the $n$th
derivative of a composition of functions; see \cite{craik2005}. Although this formula dates back to
the work of \cite{arbogast1800}, it is named after the mathematician Fa\`a di Bruno. For suitable functions $f$ and $g$, Fa\`a di Bruno's formula states that
\begin{align}
  (f\circ g)^{(n)}(x)=\sum_{k=1}^nf^{(k)}(g(x))\sum_{\bm{j}\in\mathcal{P}_{n,k}}\binom{n}{j_1,\dots,j_{n-k+1}}\prod_{l=1}^{n-k+1}\biggl(\frac{g^{(l)}(x)}{l!}\biggr)^{j_l},\label{FdB}
\end{align}
where $\binom{n}{j_1,\dots,j_n}=\frac{n!}{j_1!\cdot\ldots\cdot j_n!}$ denotes a multinomial coefficient, $\bm{j}=(j_1,\dots,j_n)\T\in\IN_0^n$, and
\begin{align}
  \mathcal{P}_{n,k}=\biggl\{\bm{j}\in\IN_0^{n-k+1}:\sum_{i=1}^{n-k+1}ij_i=n\ \text{and}\ \sum_{i=1}^{n-k+1}j_i=k\biggr\}.\label{Bell.Pnk}
\end{align}
Alternatively, one can use \emph{Bell polynomials} to reformulate (\ref{FdB}). These are defined by
\begin{align}
  B_{n,k}(x_1,\dots,x_{n-k+1})=\sum_{\bm{j}\in\mathcal{P}_{n,k}}\binom{n}{j_1,\dots,j_{n-k+1}}\prod_{l=1}^{n-k+1}\biggl(\frac{x_l}{l!}\biggr)^{j_l}.\label{Bell}
\end{align}
This implies that (\ref{FdB}) can be written as
\begin{align}
  (f\circ g)^{(n)}(x)=\sum_{k=1}^nf^{(k)}(g(x))B_{n,k}(g^\prime(x),g^{\prime\prime}(x),\dots,g^{(n-k+1)}(x))\label{FdB.Bell}
\end{align}

In the sections to come, we frequently need the following results. Here, $(x)_n=x\cdot(x-1)\cdot\ldots\cdot(x-n+1)$ denotes the \emph{falling factorial}, and $s(n,k)$ and $S(n,k)$ denote the \emph{Stirling numbers of the first} and \emph{second kind}, respectively, given by the recurrence relations
    \begin{align*}
      s(n+1,k)&=s(n,k-1)-ns(n,k),\\
      S(n+1,k)&=S(n,k-1)+kS(n,k),
    \end{align*}
    for all $k\in\IN$, $n\in\IN_0$, with $s(0,0)=S(0,0)=1$ and
    $s(n,0)=s(0,n)=S(n,0)=S(0,n)=0$ for all $n\in\IN$. Note that for $n\in\IN$ (in particular $n\neq0$), the Stirling numbers of the first kind satisfy
    \begin{align}
      (x)_n=\sum_{j=1}^ns(n,j)x^j.\label{Stirling.first}
    \end{align}
\begin{lemma}\label{lem.Bell}
  Let $B_{n,k}$ be the Bell polynomial as in (\ref{Bell}) and $n\in\IN$. Then
\begin{enumerate}[label=(\arabic*),leftmargin=*,align=left,topsep=\mytopsep,itemsep=\myitemsep]
  \item\label{lem.Bell.1} For $\bm{j}\in\mathcal{P}_{n,k}$, $\sum_{l=1}^{n-k+1}(x-l)j_l=xk-n$;
  \item\label{lem.Bell.2} $B_{n,k}(x,\dots,x)=S(n,k)x^k$, $k\in\{0,\dots,n\}$;
  \item\label{lem.Bell.3} $B_{n,k}(-x,\dots,(-1)^{n-k+1}x)=(-1)^nS(n,k)x^k$, $k\in\{0,\dots,n\}$;
  \item\label{lem.Bell.4} $\sign\bigl(B_{n,k}(g^\prime(x),g^{\prime\prime}(x),\dots,g^{(n-k+1)}(x))\bigr)=(-1)^{n-k}$ for all $x$ if $g^\prime$ is completely monotone.
  \end{enumerate}
\end{lemma}

\begin{proposition}\label{Bell.id}
   Let
   \begin{align*}
     s_{nk}(x)=\sum_{l=k}^ns(n,l)S(l,k)x^l=(-1)^n\sum_{l=k}^n\lvert s(n,l)\rvert S(l,k)(-x)^l.
   \end{align*}
   Then
   \begin{enumerate}[label=(\arabic*),leftmargin=*,align=left,topsep=\mytopsep,itemsep=\myitemsep]
   \item\label{Bell.id.1} For all $k\in\{0,\dots,n\}$, $B_{n,k}((x)_1y^{x-1},\dots,(x)_{n-k+1}y^{x-(n-k+1)})=y^{xk-n}s_{nk}(x)$;
   \item\label{Bell.id.2} $\sum_{k=1}^{n}(-1)_ks_{nk}(x)=(-x)_n$;
   \item\label{Bell.id.3} If $x\in(0,1]$, $\sign(s_{nk}(x))=(-1)^{n-k}$.
   \end{enumerate}
 \end{proposition}

\subsection{The main result}\label{sec.main}
We are now able to derive a general formula for the derivatives of the inner
generators and also for the density of nested Archimedean copulas of Type
(\ref{nac}). It will follow from Fa\`a di Bruno's formula that the derivatives
of the inner generators $\psi_{0s}(t;v_0)$ are the inner generators themselves
times a polynomial in $-v_0$. The product of these
derivatives can then be computed as a Cauchy
product. Interpreting the appearing quantities correctly allows us to compute
the expectation with respect to $F_0$ via the derivatives of $\psi_0$. This
solves all three of the above challenges.
\begin{theorem}\label{main.theorem}
  Let $\psi_s\in\Psi_\infty$, $s\in\{0,\dots,d_0\}$, such that $\cpsi_{0s}$ has completely monotone derivative for all $s\in\{1,\dots,d_0\}$.
  \begin{enumerate}[label=(\arabic*),leftmargin=*,align=left,topsep=\mytopsep,itemsep=\myitemsep]
  \item\label{main.theorem.1} For all $n\in\IN$,
    \begin{align}
      \psis{0s}{(n)}(t;v_0)=\psi_{0s}(t;v_0)\sum_{k=1}^na_{s,nk}(t)(-v_0)^k,\label{psi0sder}
    \end{align}
    where
    \begin{align}
      a_{s,nk}(t)=B_{n,k}(\cpsi_{0s}^\prime(t),\dots,\cpsi_{0s}^{(n-k+1)}(t))\label{as.nk}
    \end{align}
    with $\sign(a_{s,nk}(t))=(-1)^{n-k}$ and if $\psi_s=\psi_0$ and $n=k=1$ then $a_{s,nk}(t)=1$ for all $t$.
  \item\label{main.theorem.2} The density of (\ref{nac}) is given by
    \begin{align}
       c(\bm{u})=\biggl(\,\sum_{k=d_0}^d\var{b}{d_0}{\bm{d},k}(\bm{t}(\bm{u}))\psis{0}{(k)}(t(\bm{u}))\biggr)\cdot\prod_{s=1}^{d_0}\prod_{j=1}^{d_s}(\psiis{s})^\prime(u_{sj}),\label{dnacop}
    \end{align}
    where
    \begin{align}
      \bm{t}(\bm{u})&=(t_1(\bm{u}_1),\dots,t_{d_0}(\bm{u}_{d_0}))\T,\notag\\
      \var{b}{d_0}{\bm{d},k}(\bm{t}(\bm{u}))&=\sum_{\bm{j}\in\var{\mathcal{Q}}{d_0}{\bm{d},k}}\prod_{s=1}^{d_0}a_{s,d_sj_s}(t_s(\bm{u}_s)),\label{bk}\\
      t(\bm{u})&=\psiis{0}(C(\bm{u})),\notag
    \end{align}
    with $\bm{d}=(d_1,\dots,d_{d_0})\T$ and
    \begin{align*}
      \var{\mathcal{Q}}{d_0}{\bm{d},k}=\biggl\{\bm{j}\in\IN^{d_0}:\sum_{s=1}^{d_0}j_s=k,\ j_s\le d_s,\ s\in\{1,\dots,d_0\}\biggr\};
    \end{align*}
    that is, $\var{b}{d_0}{\bm{d},k}$ is a coefficient in the Cauchy product of the polynomials $\sum_{k=1}^{d_s}a_{s,d_sk}(t)(-v_0)^k$.
  \end{enumerate}
\end{theorem}
\begin{remark}\label{rem.der}
  \myskip
  \begin{enumerate}[label=(\arabic*),leftmargin=*,align=left,topsep=\mytopsep,itemsep=\myitemsep]
  \item We see from Theorem \ref{main.theorem} Part \ref{main.theorem.1} that
    all derivatives of $\psi_{0s}(t;v_0)$ are of similar form in $v_0$, namely
    $\psi_{0s}(t;v_0)$ times a polynomial in $-v_0$ where the coefficients
    $a_{s,d_sk}(t_s(\bm{u}_s))$ are the Bell polynomials evaluated at the
    derivatives of the nodes $\cpsi_{0s}$. This structure is crucial for solving
    Challenge 3 since it allows one to compute the expectation with respect to
    $F_0$ explicitly.
  \item We see from Theorem \ref{main.theorem} Part \ref{main.theorem.2} how the (log-)density can in general be evaluated. It involves the sign-adjusted derivatives of $\psi_0$ which are known in many cases; see \cite{hofertmaechlermcneil2012a}. Furthermore, the quantities $\var{b}{d_0}{\bm{d},k}$, $k\in\{d_0,\dots,d\}$, have to be computed. The remaining parts are comparably trivial to obtain.
  \item If there are \emph{degenerate} child copulas, that is, there exists a subset $\mathcal{S}$ of indices such that $d_s=1$ for all $s\in\mathcal{S}$, then a straightforward application of Theorem \ref{main.theorem} \ref{main.theorem.1} shows that
  \begin{align*}
       c(\bm{u})=\biggl(\,\sum_{k=d_0^\prime}^{d-d_{\mathcal{S}}}\var{b}{d_0^\prime}{\bm{d}^\prime,k}(\bm{t}(\bm{u}))\psis{0}{(k+d_{\mathcal{S}})}(t(\bm{u}))\biggr)\cdot\prod_{s=1}^{d_0}\prod_{j=1}^{d_s}(\psiis{s})^\prime(u_{sj}),
  \end{align*}
  where $d_{\mathcal{S}}= \sum_{s \in \mathcal{S}} d_s = \vert \mathcal{S} \vert$, $d_0^\prime = d_0-d_{\mathcal{S}}$ and $\bm{d}^\prime$ is the vector containing all the dimensions $d_s$ for $s\notin\mathcal{S}$.
  \end{enumerate}
\end{remark}

\section{Example families and transformations}\label{sec.ex}
\subsection{Tilted outer power families, Clayton and Gumbel copulas}
In order to construct and sample new nested Archimedean copulas it turns out to be useful to consider certain generator transformations; see \cite{hofert2010c} for more details. One such transformation leads to \emph{tilted outer power generators}
\begin{align}
  \psi_s(t)=\psi((c^{\theta_s}+t)^{1/\theta_s}-c),\label{top.gen}
\end{align}
for a generator $\psi\in\Psi_\infty$, $c\in[0,\infty)$, $\theta_s\in[1,\infty)$,
and $s\in\{0,\dots,d_0\}$. Note that generators of this form are elements of
$\Psi_\infty$. %
It follows from Equation (\ref{FdB.Bell}) and Proposition \ref{Bell.id} Part \ref{Bell.id.1} (with $x=1/\theta_0$ and $y=c^{\theta_0}+t$) that the derivatives of $\psi_0$ are
\begin{align}
  \psis{0}{(n)}(t)=\sum_{k=1}^n\psi^{(k)}((c^{\theta_0}+t)^{1/\theta_0}-c)(c^{\theta_0}+t)^{k/\theta_0-n}s_{nk}(1/\theta_0)\label{psi0dertop}
\end{align}

For nesting generators of Type (\ref{top.gen}), the nodes are given by
\begin{align*}
  \cpsi_{0s}(t)=(c^{\theta_s}+t)^{\alpha_s}-c^{\theta_0},\quad\alpha_s=\theta_0/\theta_s.
\end{align*}
This implies that tilted outer power generators of Type (\ref{top.gen}) fulfill the sufficient nesting condition if $\theta_0\le\theta_s$. Furthermore,
\begin{align*}
  \cpsi_{0s}^{(k)}(t)=(\alpha_s)_k(c^{\theta_s}+t)^{\alpha_s-k},\ k\in\IN.
\end{align*}
By Proposition \ref{Bell.id} Part \ref{Bell.id.1} (with $x=\alpha_s$ and $y=c^{\theta_s}+t$), this implies that
\begin{align*}
  a_{s,nk}(t)&=B_{n,k}(\cpsi_{0s}^\prime(t),\dots,\cpsi_{0s}^{(n-k+1)}(t))=(c^{\theta_s}+t)^{\alpha_sk-n}s_{nk}(\alpha_s).
\end{align*}
By Theorem \ref{main.theorem} Part \ref{main.theorem.1}, this implies that the inner generator
\begin{align*}
  \psi_{0s}(t;v_0)=\exp\bigl(-v_0((c^{\theta_s}+t)^{\alpha_s}-c^{\theta_0})\bigr),
\end{align*}
has derivatives
\begin{align}
  \psi_{0s}^{(n)}(t;v_0)&=\psi_{0s}(t;v_0)\sum_{k=1}^na_{s,nk}(t)(-v_0)^k=\psi_{0s}(t;v_0)\sum_{k=1}^n(c^{\theta_s}+t)^{\alpha_sk-n}s_{nk}(\alpha_s)(-v_0)^k.\label{psi0sder.top}
\end{align}

Note that $(\psiis{s})^\prime(u)=\theta_s(\psii)^\prime(u)(c+\psii(u))^{\theta_s-1}$.
By Equation (\ref{psi0dertop}), Theorem \ref{main.theorem} Part \ref{main.theorem.2}, and slight simplifications, we thus obtain
\begin{align}
  c(\bm{u})&=\biggl(\,\sum_{k=d_0}^d\var{b}{d_0}{\bm{d},k}(\bm{t}(\bm{u}))\biggl(\,\sum_{j=1}^k\psi^{(j)}\bigl((c^{\theta_0}+t(\bm{u}))^{1/\theta_0}-c\bigr)(c^{\theta_0}+t(\bm{u}))^{j/\theta_0-k}s_{kj}(1/\theta_0)\biggr)\biggr)\notag\\
&\phantom{={}}\phantom{={}} \cdot\prod_{s=1}^{d_0}\theta_s^{d_s}\prod_{j=1}^{d_s}(\psii)^\prime(u_{sj})(c+\psii(u_{sj}))^{\theta_s-1}.\label{c.top}
\end{align}

\begin{remark}[Clayton and Gumbel copulas]
  \myskip
  \begin{enumerate}[label=(\arabic*),leftmargin=*,align=left,topsep=\mytopsep,itemsep=\myitemsep]
  \item By taking $\psi(t)=1/(1+t)$ and $c=1$ we see that the tilted outer power generator (\ref{top.gen}) is $\psi_s(t)=(1+t)^{-1/\theta_s}$, that is, a generator of the Clayton family. As a special case of this section, we thus obtain the inner generator derivatives and the densities of nested Clayton copulas. Concerning the former, we obtain from (\ref{psi0sder.top}) that
    \begin{align*}
      \psi_{0s}^{(n)}(t;v_0)=\psi_{0s}(t;v_0)\sum_{k=1}^ns_{nk}(\alpha_s)(1+t)^{\alpha_sk-n}(-v_0)^k.
    \end{align*}
    Concerning the latter, plugging in the corresponding quantities in (\ref{c.top}) and simplifying the terms (in particular, the power of $1+\psiis{0}(C(\bm{u}))$ can be taken out of the inner sum), we obtain
    \begin{align*}
      c(\bm{u})&=\biggl(\,\sum_{k=d_0}^d(-1)^{d-k}\var{b}{d_0}{\bm{d},k}(\bm{t}(\bm{u}))(1+t(\bm{u}))^{-(k+1/\theta_0)}\sum_{j=1}^{k}(-1)^{k-j}s_{kj}(1/\theta_0)\biggr)\\
      &\phantom{={}}\cdot\prod_{s=1}^{d_0}\theta_s^{d_s}\biggl(\,\prod_{j=1}^{d_s}u_{sj}\biggr)^{-(1+\theta_s)}.
    \end{align*}
    By Proposition \ref{Bell.id} Part \ref{Bell.id.3}, we can further simplify this expression and obtain
    \begin{align*}
      c(\bm{u})&=(-1)^d\biggl(\,\sum_{k=d_0}^d\var{b}{d_0}{\bm{d},k}(\bm{t}(\bm{u}))(-1/\theta_0)_k(1+t(\bm{u}))^{-(k+1/\theta_0)}\biggr)\\
&\phantom{={}}\cdot\prod_{s=1}^{d_0}\theta_s^{d_s}\biggl(\,\prod_{j=1}^{d_s}u_{sj}\biggr)^{-(1+\theta_s)}
    \end{align*}
    for the density of nested Clayton copulas of Type (\ref{nac}). This formula also follows directly from Theorem \ref{main.theorem} Part \ref{main.theorem.2} by plugging in the generator derivatives $\psis{0}{(k)}(t)=(-1/\theta_0)_k(1+t)^{-(k+1/\theta_0)}$ and simplifying the expressions.
  \item Interestingly, also the inner generator derivatives and the densities of nested Gumbel copulas of Type (\ref{nac}) follow as special case of nested tilted outer power families. To see this take $\psi(t)=\exp(-t)$ (the generator of the independence copula) and consider a zero tilt (so $c=0$). It follows from (\ref{psi0sder.top}) that
    \begin{align*}
      \psi_{0s}^{(n)}(t;v_0)=\psi_{0s}(t;v_0)\sum_{k=1}^ns_{nk}(\alpha_s)t^{\alpha_sk-n}(-v_0)^k.
    \end{align*}
    Concerning the density, a short calculation shows that
    \begin{align*}
      c(\bm{u})&=(-1)^d\frac{C(\bm{u})}{\Pi(\bm{u})}\biggl(\,\sum_{k=d_0}^d\var{b}{d_0}{\bm{d},k}(\bm{t}(\bm{u}))\biggl(\,\sum_{j=1}^{k}(-t(\bm{u})^{1/\theta_0})^js_{kj}(1/\theta_0)\biggr)\biggr)\\
&\phantom{={}}\cdot\prod_{s=1}^{d_0}\theta_s^{d_s}\biggl(\,\prod_{j=1}^{d_s}-\log u_{sj}\biggr)^{\theta_s-1},
    \end{align*}
    where $C$ is (\ref{nac}) and $\Pi$ is the independence copula (hence the product of its arguments). As before, this result can also be directly obtained from \ref{main.theorem} Part \ref{main.theorem.2} based on Gumbel's generator derivatives $\psis{0}{(k)}(t)=(\psi_0(t)/t^k)\sum_{j=1}^ks_{kj}(1/\theta_0)(-1/\theta_0)^j$ as derived in \cite{hofertmaechlermcneil2012a}.
  \end{enumerate}
\end{remark}

\subsection{Ali--Mikhail--Haq copulas}
A nested Archimedean copula of Type (\ref{nac}) with all components $C_s$, $s\in\{0,\dots,d_0\}$, belonging to the Ali--Mikhail--Haq family is a valid copula according to the sufficient nesting condition if $\theta_0\le\theta_s$ for all $s\in\{1,\dots,d_0\}$. The generator $\psi_{0s}(t;v_0)$ is given by
\begin{align*}
 \psi_{0s}(t;v_0)=\biggl(\frac{1-\theta_s}{(1-\theta_0)\exp(t)-(\theta_s-\theta_0)}\biggr)^{v_0}=\biggl(\frac{1-\theta_{0s}}{\exp(t)-\theta_{0s}}\biggr)^{v_0},
\end{align*}
where $\theta_{0s}=(\theta_s-\theta_0)/(1-\theta_0)\in[0,1)$ and $v_0\in\IN$.
It can be shown from (\ref{FdB.Bell}), Lemma \ref{lem.Bell} \ref{lem.Bell.2},
and (\ref{Stirling.first}) that
\begin{align}
\psis{0s}{(n)}(t;v_0)%
&=\psi_{0s}(t;v_0)\sum_{k=1}^{n}s_{nk}\bigl(1/(1-\theta_{0s}\exp(-t))\bigr)(-v_0)^k\label{psi0sder.AMH},
\end{align}
which reveals that $a_{s,nk}(t)=s_{nk}\bigl(1/(1-\theta_{0s}\exp(-t))\bigr)$ in (\ref{psi0sder}).

\cite{hofertmaechlermcneil2012a} showed that
\begin{align*}
  \psis{0}{(k)}(t)=(-1)^k\frac{1-\theta_0}{\theta_0}\sideset{}{_{-k}}\Li(\theta_0\exp(-t)),\ t\in(0,\infty),\ k\in\IN_0,
\end{align*}
where $\sideset{}{_{s}}\Li(z)$ denotes the \emph{polylogarithm of order $s$ at
  $z$}. It follows from Theorem \ref{main.theorem} Part \ref{main.theorem.2} that
\begin{align*}
  c(\bm{u})&=(-1)^d\frac{1-\theta_0}{\theta_0}\biggl(\,\sum_{k=d_0}^d\var{b}{d_0}{\bm{d},k}(\bm{t}(\bm{u}))(-1)^k\sideset{}{_{-k}}\Li\bigl(\theta_0\exp(-t(\bm{u}))\bigr)\biggr)\\
&\phantom{={}}\cdot\prod_{s=1}^{d_0}(1-\theta_s)^{d_s}\biggl(\,\prod_{j=1}^{d_s}u_{sj}(1-\theta_s(1-u_{sj}))\biggr)^{-1}.
\end{align*}

\subsection{Joe copula}
\subsubsection*{The inner generator and its derivatives}
Nested Joe copulas of Type (\ref{nac}) are valid copulas if $\theta_0\le\theta_s$ for all $s\in\{1,\dots,d_0\}$. The generator $\psi_{0s}(t;v_0)$ is given by
\begin{align}
  \psi_{0s}(t;v_0)=(1-(1-\exp(-t))^{\alpha_s})^{v_0},\label{psi0s.Joe}
\end{align}
where $\alpha_s=\theta_0/\theta_s$, $s\in\{1,\dots,d_0\}$, and $v_0\in\IN$.
A rather lengthy calculations shows that
\begin{align*}
  \psis{0s}{(n)}(t;v_0)&=\psi_{0s}(t;v_0)(-1)^{n}\sum_{m=1}^{n}S(n,m)\biggl(-\frac{e^{-t}}{1-e^{-t}}\biggr)^m\sum_{k=1}^mv_0^k\sum_{l=k}^ms(l,k)s_{ml}(\alpha_s)\biggl(\frac{x}{1+x}\biggr)^l\\
  &=\psi_{0s}(t;v_0)\sum_{k=1}^na_{s,nk}(t)(-v_0)^k
\end{align*}
for
\begin{align*}
  a_{s,nk}(t)%
  &=\biggl((-1)^{n-k}\sum_{m=k}^{n}S(n,m)\biggl(-\frac{\exp(-t)}{1-\exp(-t)}\biggr)^m\sum_{l=k}^ms(l,k)s_{ml}(\alpha_s)\biggl(\frac{\var{\psi}{\text{J}}{1/\alpha_s}(t)-1}{\var{\psi}{\text{J}}{1/\alpha_s}(t)}\biggr)^l\biggr),
\end{align*}
where $\var{\psi}{\text{J}}{1/\alpha_s}(t)=1-(1-\exp(-t))^{\alpha_s}$ denotes Joe's generator with parameter $1/\alpha_s=\theta_s/\theta_0$. This is precisely the form as given in (\ref{psi0sder}).

It follows from \cite{hofertmaechlermcneil2012a} that
\begin{align*}
  \psis{0}{(k)}(t)=(-1)^k\frac{(1-\exp(-t))^{1/\theta_0}}{\theta_0}\var{P}{\text{J}}{k,\theta_0}\biggl(\frac{\exp(-t)}{1-\exp(-t)}\biggr),\ t\in(0,\infty),\ n\in\IN,
\end{align*}
where $\var{P}{\text{J}}{k,\theta_0}(x)=\sum_{l=1}^kS(k,l)(l-1-1/\theta_0)_{l-1}x^l$. We obtain from Theorem \ref{main.theorem} Part \ref{main.theorem.2} that
\begin{align*}
  c(\bm{u})&=\frac{(-1)^d}{\theta_0}\bigr(1-\exp(-t(\bm{u}))\bigr)^{1/\theta_0}\biggl(\,\sum_{k=d_0}^d(-1)^k\var{b}{d_0}{\bm{d},k}(\bm{t}(\bm{u}))\var{P}{\text{J}}{k,\theta_0}\biggl(\frac{\exp(t(\bm{u}))}{1-\exp(-t(\bm{u}))}\biggr)\biggr)\\
&\phantom{={}}\cdot\prod_{s=1}^{d_0}\theta_s^{d_s}\prod_{j=1}^{d_s}\frac{(1-u_{sj})^{\theta_s-1}}{1-(1-u_{sj})^{\theta_s}}.
\end{align*}
Note that $\exp(-t(\bm{u}))=\prod_{s=1}^{d_0}\bigl(1-(1-C_s(\bm{u}_s))^{\theta_0}\bigr)$.

\subsection{Frank Copula}
Nested Frank copulas of Type (\ref{nac}) are valid copulas according to the sufficient nesting condition if $\theta_0\le\theta_s$ for all $s\in\{1,\dots,d_0\}$. The generator $\psi_{0s}(t;v_0)$ is given by
\begin{align*}
  \psi_{0s}(t;v_0)=\biggl(\frac{1-(1-p_s\exp(-t))^{\alpha_s}}{p_0}\biggr)^{v_0},
\end{align*}
where $\alpha_s=\theta_0/\theta_s$, $p_j=1-e^{-\theta_j}$, $j\in\{0,s\}$, and $v_0\in\IN$. Note that this inner generator is a shifted (and appropriately scaled) inner Joe generator, that is,
\begin{align*}
  \psi_{0s}(t;v_0)=\frac{\var{\psi}{\text{J}}{0s}(h+t;v_0)}{\var{\psi}{\text{J}}{0s}(h;v_0)},
\end{align*}
where $h=-\log p_s$; see \cite[p.\ 104]{hofert2010c} for more details about such generators. In particular, with the representation for the generator derivatives for the inner Joe generator, this implies that
\begin{align*}
  \psis{0s}{(n)}(t;v_0)&=\frac{\var{\psi}{\text{J}}{0s}^{(n)}(h+t;v_0)}{\var{\psi}{\text{J}}{0s}(h;v_0)}=\frac{\var{\psi}{\text{J}}{0s}(h+t;v_0)}{\var{\psi}{\text{J}}{0s}(h;v_0)}\sum_{k=1}^n\var{a}{\text{J}}{s,nk}(t+h)(-v_0)^k\\
  &=\psi_{0s}(t;v_0)\sum_{k=1}^n\var{a}{\text{J}}{s,nk}(t+h)(-v_0)^k
\end{align*}
and thus that $a_{s,nk}(t)=\var{a}{\text{J}}{s,nk}(t+h)$, that is, the coefficients of the polynomial in $-v_0$ for the derivatives of the inner Frank generator are the ones of the inner Joe generator, appropriately shifted.

It follows from \cite{hofertmaechlermcneil2012a} that
\begin{align*}
  \psis{0}{(k)}(t)=(-1)^k\frac{1}{\theta_0}\sideset{}{_{-(k-1)}}\Li(p_0\exp(-t)),\ t\in(0,\infty),\ k\in\IN_0.
\end{align*}
Theorem \ref{main.theorem} Part \ref{main.theorem.2} then implies that
\begin{align*}
  c(\bm{u})&=(-1)^d\biggl(\,\sum_{k=d_0}^d\var{b}{d_0}{\bm{d},k}(\bm{t}(\bm{u}))(-1)^k\sideset{}{_{-(k-1)}}\Li\bigl(p_0\exp(-t(\bm{u}))\bigr)\biggr)\\
&\phantom{={}}\cdot\prod_{s=1}^{d_0}\theta_s^{d_s}\prod_{j=1}^{d_s}\frac{\exp(-\theta_su_{sj})}{1-\exp(-\theta_su_{sj})}.
\end{align*}

\subsection{A nested Ali--Mikhail--Haq $\circ$ Clayton copula}
If $\psi_0$ is the generator of an Ali--Mikhail--Haq copula and $\psi_s$, $s\in\{1,\dots,d_0\}$, generate Clayton copulas, then \cite[p.\ 115]{hofert2010c} showed that the sufficient nesting condition holds if $\theta_s\in[1,\infty)$, $s\in\{1,\dots,d_0\}$, so one can build nested Archimedean copulas of Type (\ref{nac}) with the root copula $C_0$ being of Ali--Mikhail--Haq and the child copulas $C_s$ being of Clayton type under this condition (referred to as Ali--Mikhail--Haq $\circ$ Clayton copulas). In this case, a short calculation shows that
\begin{align*}
  \psi_{0s}(t;v_0)=\psi((1+t)^{1/\theta_s}-1)
\end{align*}
for $\psi(t)=(1+(1-\theta_0)t)^{-v_0}$. We can thus apply (\ref{psi0dertop}) with $c=1$ to see that
\begin{align}
  \psis{0s}{(n)}(t;v_0)=\sum_{j=1}^n\psi^{(j)}((1+t)^{1/\theta_s}-1)(1+t)^{j/\theta_s-n}s_{nj}(1/\theta_s),\label{psi0sder.AMH.C}
\end{align}
where
\begin{align*}
  \psi^{(j)}((1+t)^{1/\theta_s}-1)=(1-\theta_0)^j\psi_{0s}(t;v_0)\psi_{0s}(t;j)\sum_{k=1}^js(j,k)(-v_0)^k.
\end{align*}
Plugging this result into (\ref{psi0sder.AMH.C}) and interchanging the order of the two summations, we obtain
\begin{align*}
  \psis{0s}{(n)}(t;v_0)%
  &=\psi_{0s}(t;v_0)\sum_{k=1}^n\biggl(\,\sum_{j=k}^{n}s(j,k)s_{nj}(1/\theta_s)\psi_{0s}(t;j)(1-\theta_0)^j(1+t)^{j/\theta_s-n}\biggr)(-v_0)^k,
\end{align*}
which provides the structure of the coefficients $a_{s,nk}$ in (\ref{psi0sder}), namely,
\begin{align*}
  a_{s,nk}(t)=\sum_{j=k}^{n}s(j,k)s_{nj}(1/\theta_s)\biggl(\frac{1-\theta_0}{\theta_0+(1-\theta_0)(1+t)^{1/\theta_s}}\biggr)^j(1+t)^{j/\theta_s-n}.
\end{align*}

It is clear from Theorem \ref{main.theorem} Part \ref{main.theorem.2} that the density for the nested Ali--Mikhail--Haq $\circ$ Clayton copula basically consists of the corresponding pieces of the Ali--Mikhail--Haq and the Clayton density we have already seen earlier. It is given by
\begin{align*}
  c(\bm{u})%
&=(-1)^d\frac{1-\theta_0}{\theta_0}\biggl(\,\sum_{k=d_0}^d\var{b}{d_0}{\bm{d},k}(\bm{t}(\bm{u}))(-1)^k\sideset{}{_{-k}}\Li\bigl(\theta_0\exp(-t(\bm{u}))\bigr)\biggr)\prod_{s=1}^{d_0}\theta_s^{d_s}\biggl(\,\prod_{j=1}^{d_s}u_{sj}\biggr)^{-(1+\theta_s)}.
\end{align*}
Note that $\displaystyle\exp(t(\bm{u}))=\prod_{s=1}^{d_0}\frac{C_s(\bm{u}_s)}{1-\theta_0(1-C_s(\bm{u}_s))}$.

\section{Numerical evaluation}\label{sec.num}
\subsection{The log-density}
In statistical applications one typically aims at computing the
log-density. From a numerical point of view, this is typically not as trivial as
computing the density and taking the logarithm afterwards. Often, the density
can not be computed without running into numerical problems, hence taking the
logarithm of the density faces the same problem. However, an intelligent
implementation of the log-density is possible (and often even required), see the implementation in the \textsf{R} package \texttt{copula}.

We now briefly explain how one can efficiently compute the log-density of a nested Archimedean copula of Type (\ref{nac}). Recall from (\ref{dnacop}) that
\begin{align*}
  c(\bm{u})=\biggl(\,\sum_{k=d_0}^d\var{b}{d_0}{\bm{d},k}(\bm{t}(\bm{u}))\psis{0}{(k)}(t(\bm{u}))\biggr)\cdot\prod_{s=1}^{d_0}\prod_{j=1}^{d_s}(\psiis{s})^\prime(u_{sj}).
\end{align*}
Let us first think about the signs of the terms $\var{b}{d_0}{\bm{d},k}(\bm{t}(\bm{u}))$, $k\in\{d_0,\dots,d\}$. By Theorem \ref{main.theorem} \ref{main.theorem.1} we know that $\sign\bigl(a_{s,d_sj_s}(t_s(\bm{u}_s))\bigr)=(-1)^{d_s-j_s}$, thus
\begin{align*}
  \sign\prod_{s=1}^{d_0}a_{s,d_sj_s}(t_s(\bm{u}_s))=(-1)^{\sum\limits_{s=1}^{d_0}d_s-\sum\limits_{s=1}^{d_0}j_s}.
\end{align*}
Recall from (\ref{bk}) the structure of $\var{b}{d_0}{\bm{d},k}(\bm{t}(\bm{u}))$, which is the sum in $\bm{j}\in\var{\mathcal{Q}}{d_0}{\bm{d},k}$ over $\prod_{s=1}^{d_0}a_{s,d_sj_s}(t_s(\bm{u}_s))$. For such $\bm{j}$, it follows from the definition of $\var{\mathcal{Q}}{d_0}{\bm{d},k}$ that $\sum_{s=1}^{d_0}j_s=k$. Furthermore, note that $\sum_{s=1}^{d_0}d_s=d$, hence
\begin{align*}
  \sign\var{b}{d_0}{\bm{d},k}(\bm{t}(\bm{u}))=(-1)^{d-k}.
\end{align*}
This implies that
\begin{align}
  c(\bm{u})=\biggl(\,\sum_{k=d_0}^d(-1)^{d-k}\var{b}{d_0}{\bm{d},k}(\bm{t}(\bm{u}))(-1)^k\psis{0}{(k)}(t(\bm{u}))\biggr)\cdot\prod_{s=1}^{d_0}\prod_{j=1}^{d_s}(-\psiis{s})^\prime(u_{sj})\label{dnacop.}
\end{align}
where we note that $\prod_{s=1}^{d_0}\prod_{k=1}^{d_s}(-1)=(-1)^d$.

We see from (\ref{dnacop.}) that all appearing quantities are positive which is quite convenient for computing the log-density
\begin{align*}
  \log c(\bm{u})=\log\biggl(\,\sum_{k=d_0}^d(-1)^{d-k}\var{b}{d_0}{\bm{d},k}(\bm{t}(\bm{u}))(-1)^k\psis{0}{(k)}(t(\bm{u}))\biggr)+\sum_{s=1}^{d_0}\sum_{j=1}^{d_s}\log(-\psiis{s})^\prime(u_{sj}).
\end{align*}
Since the latter double sum is typically trivial to compute, let us focus on the first sum. To compute the (intelligent) logarithm of this sum, let
\begin{align*}
  x_k=\log\bigl((-1)^{d-k}\var{b}{d_0}{\bm{d},k}(\bm{t}(\bm{u}))\bigr)+\log\bigl((-1)^k\psis{0}{(k)}(t(\bm{u}))\bigr),\ k\in\{d_0,\dots,d\},
\end{align*}
and note that
\begin{align*}
 \log\sum_{k=d_0}^d(-1)^{d-k}\var{b}{d_0}{\bm{d},k}(\bm{t}(\bm{u}))(-1)^k\psis{0}{(k)}(t(\bm{u}))&=\log\sum_{k=d_0}^{d}\exp(x_k)\\
&=x_{\text{max}}+\log\sum_{k=d_0}^{d}\exp(x_k-x_{\text{max}}),
\end{align*}
where $x_{\text{max}}=\max_{d_0\le k\le d}x_k$. Since all summands in the latter sum are in $(0,1]$, the corresponding logarithm can easily be computed. It remains to discuss how the $x_k$, $k\in\{d_0,\dots,d\}$, can be computed.

For computing the $x_k$, $k\in\{d_0,\dots,d\}$, efficient implementations for the functions $\log((-1)^k\psis{0}{(k)}(t))$ in the \texttt{R} package \texttt{copula} can be used. Computing the quantities $\log\bigl((-1)^{d-k}\var{b}{d_0}{\bm{d},k}(\bm{t}(\bm{u}))\bigr)$ is more challenging. Recall from (\ref{bk}) that
\begin{align*}
  \var{b}{d_0}{\bm{d},k}(\bm{t}(\bm{u}))=\sum_{\bm{j}\in\var{\mathcal{Q}}{d_0}{\bm{d},k}}\prod_{s=1}^{d_0}a_{s,d_sj_s}(t_s(\bm{u}_s)),
\end{align*}
where $a_{s,d_sj_s}(t_s(\bm{u}_s))$ is given in (\ref{as.nk}). For computing the function $s_{d_sj_s}$ that often appears in $a_{s,d_sj_s}(t_s(\bm{u}_s))$, the function \texttt{coeffG} in \texttt{copula} can be used; to be more precise, \lstinline!(-1)^(!$d_s$\lstinline!-!$j_s$\lstinline!)!\,\lstinline!* copula:::coeffG(!$d_s$\lstinline!,x)! computes $s_{d_sj_s}(\texttt{x})$. For the summation over the set $\var{\mathcal{Q}}{d_0}{\bm{d},k}$, the \textsf{R} package \texttt{partitions} provides the function \texttt{blockparts}. With \lstinline!blockparts(!$\bm{d}$\lstinline!-rep(1L,!$d_0$\lstinline!),!$k$\lstinline!-!$d_0$\lstinline!)+1L! one can then obtain a matrix with $d_0$ rows where each column gives one $\bm{j}\in\var{\mathcal{Q}}{d_0}{\bm{d},k}$.

\subsection{The -log-likelihood of two-parameter nested Gumbel copulas}
In this section, we compute the -log-likelihood (based on a sample of size $n=100$) of two nested Gumbel copulas with parameters $\theta_0$ and $\theta_1$ such that Kendall's tau equals $0.25$ and $0.5$, respectively. In order to be able to provide graphical insights, we focus on two-parameter nested copulas of the form
\begin{align}
  C(\bm{u})=C_0(u_1,C_1(u_2,\dots,u_{d}))\label{fnac}
\end{align}
where $d\in\{3,10\}$.

Note that we obtain nested Archimedean copulas of Type (\ref{fnac}) from (\ref{nac}) by artificially thinking of $u_1$ as a child copula $\psi_0(\psiis{0}(u_1))$ of dimension 1, that is, as a degenerate child copula. For such $s$, note that $a_{s,d_sj_s}(t_s(\bm{u}_s))=a_{s,11}(t_s(\bm{u}_s))$ and $t_s(\bm{u}_s)$ equals $\psiis{0}$ at the corresponding (one-dimensional) argument, which is $u_1$ in (\ref{fnac}). It follows from the last statement in Theorem \ref{main.theorem} \ref{main.theorem.1} that $a_{s,d_sj_s}(t_s(\bm{u}_s))=1$ for degenerate children. This implies that these terms drop out of the product in (\ref{bk}). The set $\var{\mathcal{Q}}{d_0}{\bm{d},k}$ shrinks accordingly since $1\le j_s\le d_s=1$ for degenerate children $s$.

Figure~\ref{fig.logL.G} displays the -log-likelihoods as level plots for the
nested Gumbel copulas as described above based on a sample of size $n=100$ (note
the restriction $\theta_0\le\theta_1$). The parameters from which the samples
were drawn and the minima based on the grid points as displayed in the wireframe
plots are included. It is interesting to see the behavior of the -log-likelihood
in the child parameter $\theta_1$ when the dimension of the child copula $C_1$
is increased (with $C_0$ and its dimension fixed). As can be seen from the
plots, the -log-likelihood is easier to minimize in $\theta_1$-direction. This
behavior was already observed by \cite{hofertmaechlermcneil2012a} for
Archimedean copulas and can be expressed by the empirical observation that the
mean squared error behaves like $1/(nd)$ which is decreasing in $d$ for fixed
$n$. To see a similar behavior here for the -log-likelihoods of the nested
Gumbel copulas is not surprising since the marginal copula for $u_1=1$ is the
Archimedean copula $C_1$. Finally, let us remark that the optimization procedure
used to generated Figure~\ref{fig.logL.G} can also be found in the \textsf{R} package
\texttt{copula} as a demo. This can directly be used for fitting a two-level
nested Archimedean copula to a real life data set. Furthermore, similar figures
as Figure~\ref{fig.logL.G} are provided in the demo for a nested Clayton copula.
\begin{figure}[htbp]
  \centering
  \includegraphics[width=0.5\textwidth]{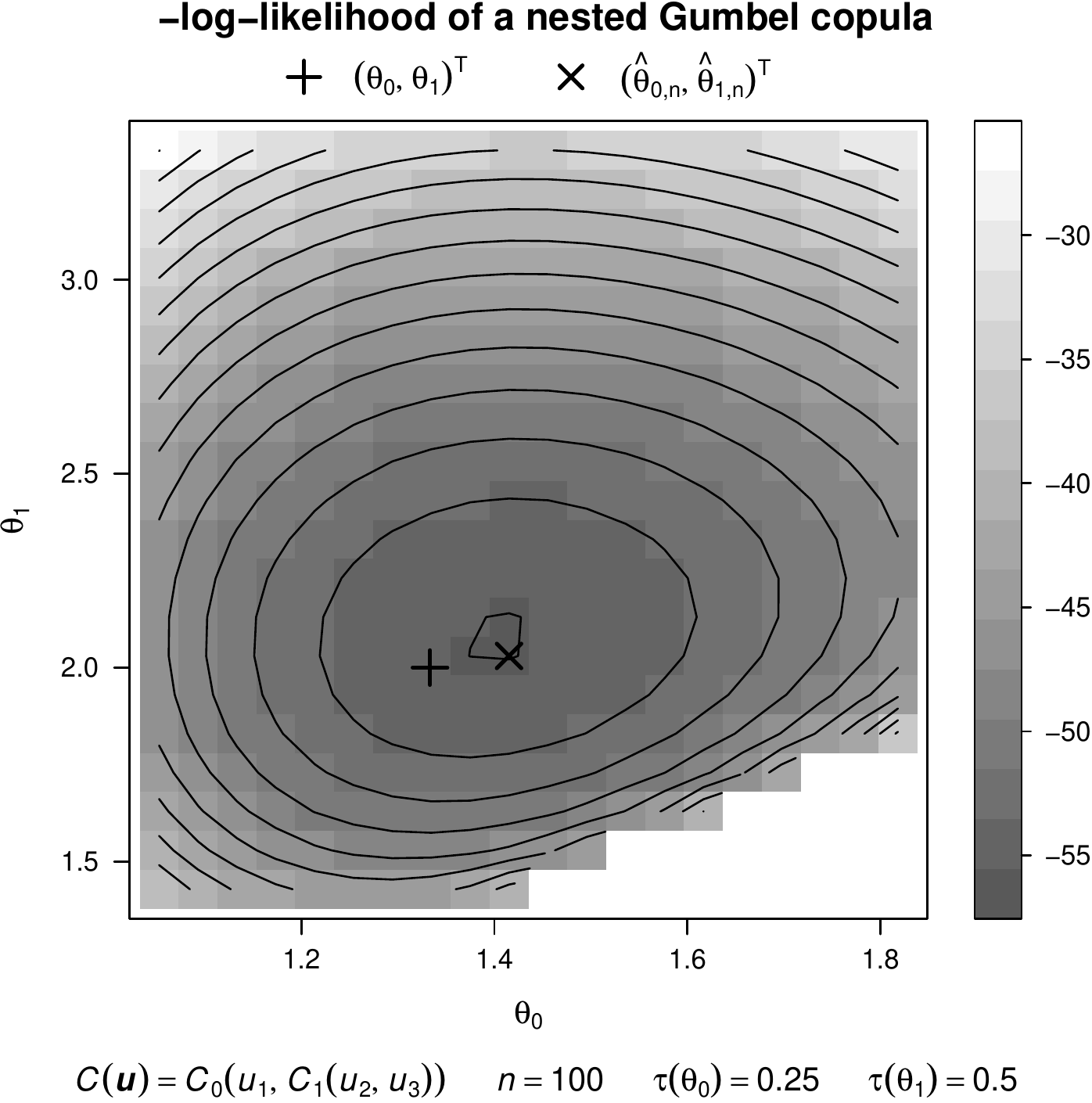}%
  \hfill
  \includegraphics[width=0.5\textwidth]{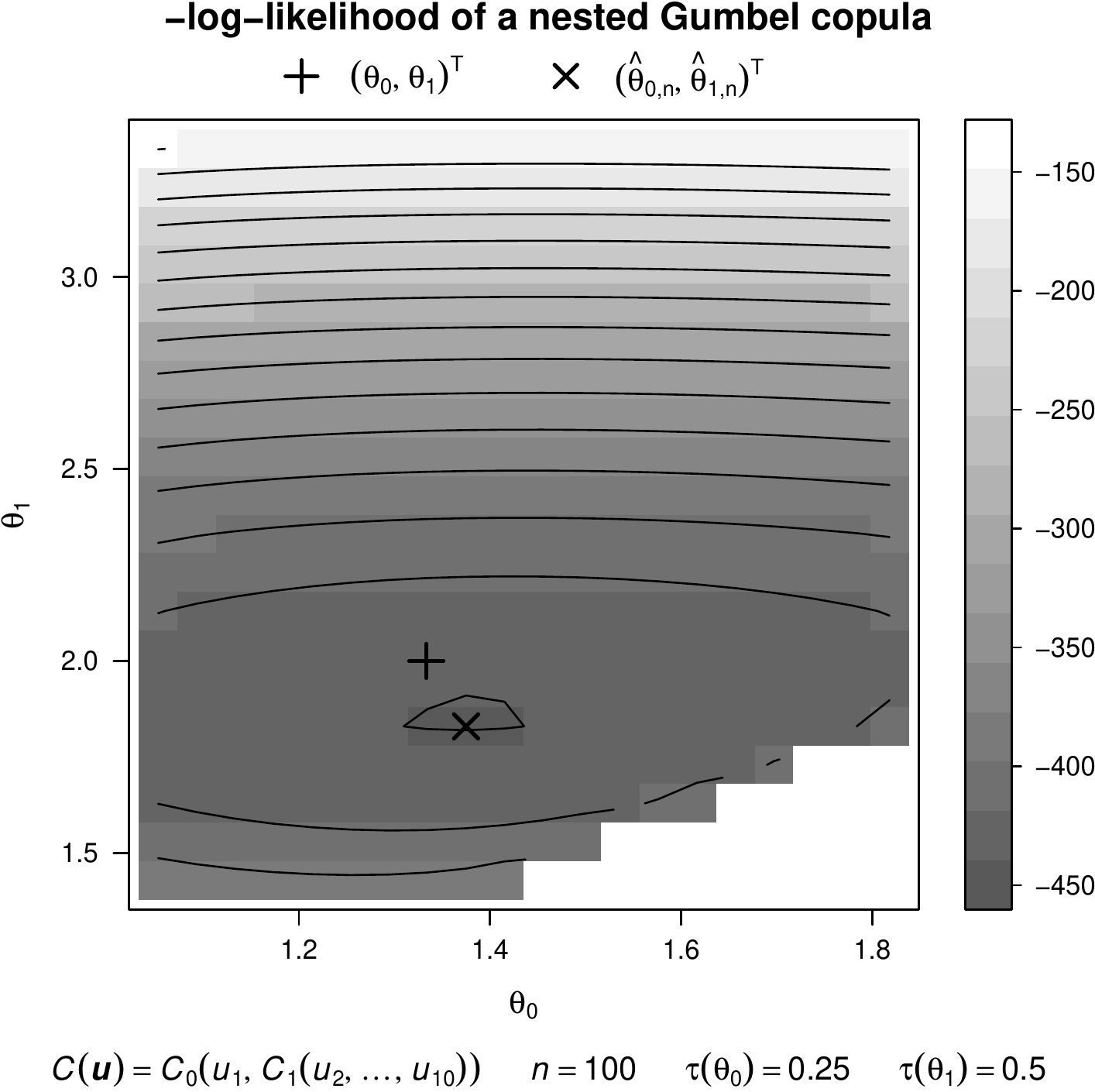}%
  \setcapwidth{\textwidth}%
  \caption{Wireframe plots of the -log-likelihood of a three-dimensional (left)
    and ten-dimensionsl (right) nested Gumbel copula $C(\bm{u})=C_0(u_1,C_1(u_2,u_3))$ with parameters $\theta_0=4/3$ (Kendall's tau equals $0.25$) and $\theta_1=2$ (Kendall's tau equals $0.5$) based on a sample of size $n=100$.}
  \label{fig.logL.G}
\end{figure}

\section{Densities for three- (and higher-) level nested Archimedean copulas}\label{sec.ext}
In this section, a density formula analogous to (\ref{dnacop}) is derived for
three-level nested Archimedean copulas and extensions to higher nesting levels
are briefly addressed.

When working with three or more nesting levels, it turns
out to be convenient to (slightly) change the notation used in the previous sections. Consider a three-level nested Archimedean copula of the form
\begin{align}
  C(\bm{u})=C_1\bigl(C_{11}(C_{111}(\bm{u}_{111}),\dots,C_{11d_{11}}(\bm{u}_{11d_{d11}})),\dots,C_{1d_1}(C_{1d_11}(\bm{u}_{1d_11}),\dots,C_{1d_1d_{1d_1}}(\bm{u}_{1d_1d_{1d_1}}))\bigr),\label{nAC3}
\end{align}
where $\bm{u}_{s_1s_2s_3}=(u_{s_1s_2s_31},\dots,u_{s_1s_2s_3d_{s_1s_2s_3}})\T$
denotes the argument of $C_{s_1s_2s_3}$ (the copula generated by
$\psi_{s_1s_2s_3}$), $d_{s_1s_2s_3}$ denotes the dimension of $C_{s_1s_2s_3}$,
and $d_{s_1s_2}$ denotes the dimension of $C_{s_1s_2}$ (the copula generated by
$\psi_{s_1s_2}$). Here and in the
following, $s_1$ always equals 1, $s_2\in\{1,\dots,d_{s_1}\}$, and
$s_3\in\{1,\dots,d_{s_1s_2}\}$. Note that it is convenient to think of
(\ref{nAC3}) as a tree; see Figure~\ref{fig.nAC3.tree}. Furthermore, let
\begin{align*}
  t_{s_1s_2s_3}(\bm{u}_{s_1s_2s_3})&=\sum_{s_4=1}^{d_{s_1s_2s_3}}\psiis{s_1s_2s_3}(u_{s_1s_2s_3s_4})=\psi_{s_1s_2s_3}^{-1}(C_{s_1s_2s_3}(\bm{u}_{s_1s_2s_3})),\\
\bm{u}_{s_1s_2}&=(\bm{u}_{s_1s_21}\T,\dots,\bm{u}_{s_1s_2d_{s_1s_2}}\T)\T,\\
  \bm{t}_{s_1s_2}(\bm{u}_{s_1s_2})&=(t_{s_1s_21}(\bm{u}_{s_1s_21}),\dots,t_{s_1s_2d_{s_1s_2}}(\bm{u}_{s_1s_2d_{s_1s_2}}))\T,\\
  C_{s_1s_2}^\ast(\bm{u}_{s_1s_2})&=C_{s_1s_2}(C_{s_1s_21}(\bm{u}_{s_1s_21}),\dots,C_{s_1s_2d_{s_1s_2}}(\bm{u}_{s_1s_2d_{s_1s_2}}))\\
&=\psi_{s_1s_2}\biggl(\,\sum_{s_3=1}^{d_{s_1s_2}}
  \cpsi_{s_1s_2,s_1s_2s_3}\bigl(t_{s_1s_2s_3}(\bm{u}_{s_1s_2s_3})\bigr)\biggr),\\
  t^\ast_{s_1s_2}(\bm{u}_{s_1s_2})&=\psiis{s_1s_2}(C_{s_1s_2}^\ast(\bm{u}_{s_1s_2}))=\sum_{s_3=1}^{d_{s_1s_2}}
  \cpsi_{s_1s_2,s_1s_2s_3}\bigl(t_{s_1s_2s_3}(\bm{u}_{s_1s_2s_3})\bigr),
\end{align*}
where
\begin{align*}
  \cpsi_{s_1s_2,s_1s_2s_3}=\psi_{s_1s_2}^{-1} \circ \psi_{s_1s_2s_3}
\end{align*}
and $C_{s_1s_2}^\ast$ denotes the (marginal) nested Archimedean
copula with root $C_{s_1s_2}$. Note that the dimension of the root copula
$C_{s_1s_2}$ of the nested Archimedean copula $C_{s_1s_2}^\ast$ is $d_{s_1s_2}$
which is in general not equal to the dimension
$d_{s_1s_2\cdot}=\sum_{s_3=1}^{d_{s_1s_2}}d_{s_1s_2s_3}$ of
$C_{s_1s_2}^\ast$. Furthermore, the root copula $C_1$ ($=C_{s_1}$) of the nested
Archimedean copula $C$ has $d_1$ ($=d_{s_1}$) arguments, the $s_2$th of which
has $d_{s_1s_2}$-many arguments. Overall, $d_{s_1\cdot\cdot}=\sum_{s_2=1}^{d_{s_1}}
\sum_{s_3=1}^{d_{s_1s_2}} d_{s_1s_2s_3}$ equals $d$, the dimension of $C$.

In order to compute the density $c$ of $C$, we use a similar idea as in Section
\ref{sec.basic}.

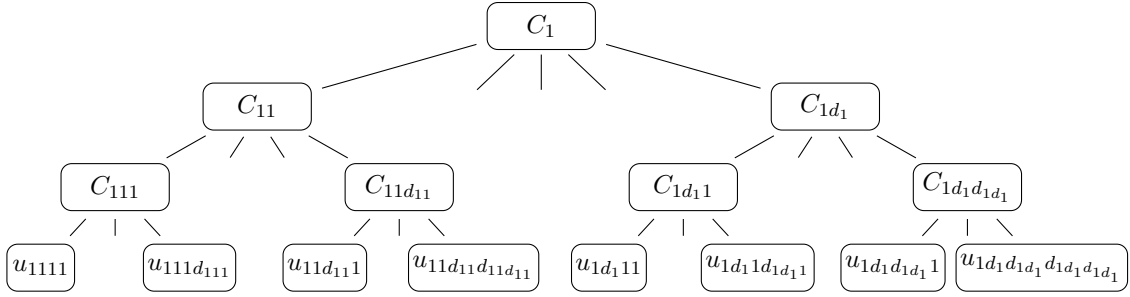
\begin{figure}[htbp]
  \centering
  \begin{tikzpicture}[
    transform shape, scale=0.89,
    grow=south,
    level 1/.style={sibling distance=84mm,level distance=12mm},
    level 2/.style={sibling distance=42mm,level distance=12mm},
    level 3/.style={sibling distance=22mm,level distance=12mm},
    edge from parent/.style={draw,shorten >=1.4mm,shorten <=1.4mm},
    every node/.style={draw,rectangle,rounded corners,inner sep=0.6mm},
    mynodestyle/.style={minimum width=16mm,minimum height=7mm},
    myleafstyle/.style={minimum width=10mm,minimum height=7mm},
    mydrawstyle/.style={dotted,thick}
    ]
    \node[mynodestyle]{$C_1$}
    child{
      node[mynodestyle]{$C_{11}$}{
        child{
          node[mynodestyle]{$C_{111}$}{
            child{
              node[myleafstyle]{$u_{1111}$}
            }
            child{
              node[myleafstyle]{$u_{111d_{111}}$}
            }
          }
        }
        child{
          node[mynodestyle]{$C_{11d_{11}}$}{
            child{
              node[myleafstyle]{$u_{11d_{11}1}$}
            }
            child{
              node[myleafstyle]{$u_{11d_{11}d_{11d_{11}}}$}
            }
          }
        }
      }
    }
    child{
      node[mynodestyle]{$C_{1d_1}$}{
        child{
          node[mynodestyle]{$C_{1d_11}$}{
            child{
              node[myleafstyle]{$u_{1d_111}$}
            }
            child{
              node[myleafstyle]{$u_{1d_11d_{1d_11}}$}
            }
          }
        }
        child{
          node[mynodestyle]{$C_{1d_1d_{1d_1}}$}{
            child{
              node[myleafstyle]{$u_{1d_1d_{1d_1}1}$}
            }
            child{
              node[myleafstyle]{$u_{1d_1d_{1d_1}d_{1d_1d_{1d_1}}}$}
            }
          }
        }
      }
    };
    \draw(-4mm,-4.25mm)--(-9.5mm,-9.5mm);%
    \draw(0mm,-4.25mm)--(0mm,-9.5mm);%
    \draw(4mm,-4.25mm)--(9.5mm,-9.5mm);%
    \draw(-44mm,-16.5mm)--(-46mm,-19.6mm);%
    \draw(-40mm,-16.5mm)--(-38mm,-19.6mm);%
    \draw(40mm,-16.5mm)--(38mm,-19.6mm);%
    \draw(44mm,-16.5mm)--(46mm,-19.6mm);%
    \draw(-63mm,-28.75mm)--(-63mm,-31.25mm);%
    \draw(-21mm,-28.75mm)--(-21mm,-31.25mm);%
    \draw(21mm,-28.75mm)--(21mm,-31.25mm);%
    \draw(63mm,-28.75mm)--(63mm,-31.25mm);%
  \end{tikzpicture}
  \setcapwidth{\textwidth}%
  \caption{Tree structure (some arguments are omitted) for a three-level nested Archimedean copula of Type (\ref{nAC3}).}
  \label{fig.nAC3.tree}
\end{figure}

By replacing $\psi_1\text{(}\!=\psi_{s_1}\text{)}=\LS[F_1]$ with the corresponding integral, we obtain
\begin{align*}
  C(\bm{u}) & =\int_0^\infty\prod_{s_2=1}^{d_1}\exp\bigl(-v_1\psiis{s_1}(C_{s_1s_2}^\ast(\bm{u}_{s_1s_2}))\bigr)\,dF_1(v_1) \\
  & =\int_0^\infty\prod_{s_2=1}^{d_1}\psi_{s_1,s_1s_2}(t_{s_1s_2}^\ast(\bm{u}_{s_1s_2});v_1)\,dF_1(v_1)
\end{align*}
where
\begin{align*}
 \psi_{s_1,s_1s_2}(t;v_1) = \exp\bigl(-v_1 \cpsi_{s_1,s_1s_2}(t)\bigr)
\end{align*}
and thus
\begin{align}
  c(\bm{u})=\int_0^\infty\prod_{s_2=1}^{d_1}\frac{\partial}{\partial\bm{u}_{s_1s_2}}\psi_{s_1,s_1s_2}(t_{s_1s_2}^\ast(\bm{u}_{s_1s_2});v_1)\,dF_1(v_1),\label{dens3l}
\end{align}
where $\frac{\partial}{\partial\bm{u}_{s_1s_2}}$ denotes the derivative with respect to all components of $\bm{u}_{s_1s_2}$ (which is a vector of length $d_{s_1s_2\cdot}$).

Similar as in Section \ref{sec.basic}, we observe the following key challenges:
\begin{enumerate}[label=\sffamily\bfseries Challenge \arabic*\ ,leftmargin=*,align=left,topsep=\mytopsep,itemsep=\myitemsep]
\item Find the derivatives in the integrand;
\item Compute their product;
\item Integrate it with respect to the mixture distribution function $F_1=\LSi[\psi_1]$.
\end{enumerate}

We will first solve Challenge 1 by considering a multivariate version of Fa\`a di
Bruno's formula. For suitable functions $f:\IR\to\IR$ and $g:\IR^n\to\IR$, it follows from
\cite{hardy2006} that
\begin{align}
  \frac{\partial}{\partial\bm{x}}f(g(\bm{x}))=\sum_{k=1}^nf^{(k)}(g(\bm{x}))\sum_{\pi:\lvert\pi\rvert=k}\prod_{B\in\pi}\frac{\partial^{\lvert
      B\rvert}}{\prod_{i\in B}\partial x_i}g(\bm{x}),\ \bm{x}=(x_1,\dots,x_n)\T,\label{FdBMulti}
\end{align}
where the last sum extends over all partitions $\pi$ of $\{1,\dots,n\}$ with $k$
elements and
the last product over all blocks $B$ of $\pi$. Observe that,  if $x_1 = \dots  =  x_n = x$, then the univariate Fa\`a di Bruno's formula (\ref{FdB.Bell}) can be restated as
\begin{align*}
(f \circ g)^{(n)}(x) = \sum_{k=1}^n f^{(k)}(g(x)) \sum_{\pi:\vert \pi \vert = k} \prod_{B \in \pi} g^{(\vert B \vert)}(x),
\end{align*}
where $\pi$ is a partition of $\{1,\dots, n \}$. Comparing this identity with (\ref{FdB.Bell}) yields
\begin{align}\label{BellCombin}
B_{n,k}(g^\prime(x), \dots , g^{(n-k+1)}(x)) =  \sum_{\pi:\vert \pi \vert = k} \prod_{B \in \pi} g^{(\vert B \vert)}(x).
\end{align}
This will be used in the following lemma, which is a special case of
(\ref{FdBMulti}) with stronger assumptions on the function $g$. It will then
lead us to a solution for Challenge 1 by choosing suitable functions $f$ and $g$.
\begin{lemma}\label{FdbMultiStrong}
Suppose there exists a partition $\{B_1,\dots,B_m\}$
of $\{1,\dots,n\}$ with
$\lvert B_l \rvert = d_l$ for $l\in \{ 1,\dots, m\}$ (with $\sum_{l=1}^md_l=n$), such that for any indices $k_1 \in B_i$ and $k_2 \in B_j$, for $i,j \in \{ 1,\dots, m \}$ with $i\neq j$, the partial derivative of $g(x_1,\dots,x_n)$ with respect to $x_{k_1}$  and $x_{k_2}$ equals zero, that is
\begin{align}
 \frac{\partial^2  }{\partial x_{k_1} \partial x_{k_2}} \, g(\bm{x})= 0,\quad\text{for all}\ k_1\in B_i,\ k_2 \in B_j,\ i,j\in\{1,\dots,m\},\ i\neq j. \label{FdBMulti.hyp1}
\end{align}
Moreover, suppose that for any $l \in \{ 1,\dots, m \}$ and any subset $B$ of $B_l$, there exist functions $h_{l1}$ and $h_{l2}$ such that
\begin{align}
  \frac{\partial^{\lvert B\rvert}}{\prod_{i\in B}\partial x_i}g(\bm{x}) = h_{l1}^{(\lvert B \rvert)}(h_{l2}(\bm{x})) \prod_{i\in B} \frac{\partial}{\partial x_i} h_{l2}(\bm{x}).\label{h1h2}
\end{align}

Then one has, for any suitable function $f$,
\begin{align*}
\frac{\partial}{\partial\bm{x}}f(g(\bm{x}))&=\biggl(\,\prod_{l=1}^m \prod_{i \in B_l} \frac{\partial}{\partial x_i} h_{l2}(\bm{x}) \biggr)  \\
& \phantom{{}={}} \cdot
 \sum_{k=1}^n f^{(k)}(g(\bm{x})) \sum_{\bm{j} \in \var{\mathcal{Q}}{m}{\bm{d},k}} \prod_{l=1}^m  B_{d_l, j_l}\bigl(h_{l1}^\prime (h_{l2}(\bm{x})), \dots , h_{l1}^{(d_l-j_l+1)}(h_{l2}(\bm{x}))\bigr),
\end{align*}
where $\var{\mathcal{Q}}{m}{\bm{d},k}$ is defined as in Theorem
\ref{main.theorem} \ref{main.theorem.2} and $\bm{d}=(d_1,\dots,d_m)\T$.
\end{lemma}

We are now in the position to solve Challenge 1. By applying Lemma
\ref{FdbMultiStrong} with $\bm{x} = \bm{u}_{s_1s_2}$, $n = d_{s_1s_2\cdot}$, $m =
d_{s_1s_2}$, $B_l=\{s_1s_2l1,\dots, s_1s_2ld_{s_1s_2l}\}$ (slightly abusing the notation), $\bm{d}_{s_1s_2} =
(d_{s_1s_21},\dots, d_{s_1s_2d_{s_1s_2}})\T$,  $f(t) = \psi_{s_1,s_1s_2}(t; v_1)$, $g(\bm{x})=t_{s_1s_2}^\ast(\bm{u}_{s_1s_2})$,
$h_{l1}(t)  = \cpsi_{s_1s_2,s_1s_2l}(t)$, and $h_{l2}(\bm{u})  =
t_{s_1s_2l}(\bm{u}_{s_1s_2l}) = \sum_{s_4=1}^{d_{s_1s_2l}} \psi^{-1}_{s_1s_2l}(u_{s_1s_2ls_4})$, the derivative in the integrand in (\ref{dens3l}) is given by
\begin{align*}
    &\phantom{{}={}}\frac{\partial}{\partial\bm{u}_{s_1s_2}}\psi_{s_1,s_1s_2}(t_{s_1s_2}^\ast(\bm{u}_{s_1s_2});v_1)\\
    &  = \biggl(\,\prod_{s_3=1}^{d_{s_1s_2}} \prod_{s_4=1}^{d_{s_1s_2s_3}} (\psi_{s_1s_2s_3}^{-1})^\prime(u_{s_1s_2s_3s_4})\biggr) \sum_{l=1}^{d_{s_1s_2\cdot}} \psi_{s_1,s_1s_2}^{(l)}(t_{s_1s_2}^\ast(\bm{u}_{s_1s_2});v_1) \\
        & \phantom{{}={}}\cdot\biggl(\, \sum_{\bm{j} \in \mathcal{Q}^{d_{s_1s_2}}_{\bm{d}_{s_1s_2},l}}
    \prod_{s_3=1}^{d_{s_1s_2}} B_{d_{s_1s_2s_3},j_{s_3}}\bigl
    ((\cpsi_{s_1s_2,s_1s_2s_3}^{(k)}(t_{s_1s_2s_3}(\bm{u}_{s_1s_2s_3})))_{k\in\{ 1,\dots, d_{s_1s_2s_3}- j_{s_3}+1\} }\bigr)\biggr)
  \end{align*}
  By applying Theorem \ref{main.theorem} \ref{main.theorem.1}, this derivative can be written as
  \begin{align*}
  &\phantom{{}={}}\frac{\partial}{\partial\bm{u}_{s_1s_2}}\psi_{s_1,s_1s_2}(t_{s_1s_2}^\ast(\bm{u}_{s_1s_2});v_1)\\
  &=\biggl(\,\prod_{s_3=1}^{d_{s_1s_2}}\prod_{s_4=1}^{d_{s_1s_2s_3}}(\psi_{s_1s_2s_3}^{-1})^\prime(u_{s_1s_2s_3s_4})\biggr)\psi_{s_1,s_2s_2}(t_{s_1s_2}^\ast(\bm{u}_{s_1s_2});v_1)\\
  &\phantom{{}={}}\cdot\sum_{l=1}^{d_{s_1s_2\cdot}}\sum_{k=1}^l a_{s_1s_2,lk}(t_{s_1s_2}^\ast(\bm{u}_{s_1s_2}))(-v_1)^k\\
        & \phantom{{}={}}\cdot\biggl(\, \sum_{\bm{j} \in \mathcal{Q}^{d_{s_1s_2}}_{\bm{d}_{s_1s_2},l}}
    \prod_{s_3=1}^{d_{s_1s_2}} B_{d_{s_1s_2s_3},j_{s_3}}\bigl
    ((\cpsi_{s_1s_2,s_1s_2s_3}^{(k)}(t_{s_1s_2s_3}(\bm{u}_{s_1s_2s_3})))_{k\in\{ 1,\dots, d_{s_1s_2s_3}- j_{s_3}+1\} }\bigr)\biggr)
    \end{align*}
  Interchanging the order of summation in the double sum yields
  \begin{align*}
  &\phantom{{}={}}\frac{\partial}{\partial\bm{u}_{s_1s_2}}\psi_{s_1,s_1s_2}(t_{s_1s_2}^\ast(\bm{u}_{s_1s_2});v_1)\\
  &=\biggl(\,\prod_{s_3=1}^{d_{s_1s_2}}\prod_{s_4=1}^{d_{s_1s_2s_3}}(\psi_{s_1s_2s_3}^{-1})^\prime(u_{s_1s_2s_3s_4})\biggr)\psi_{s_1,s_2s_2}(t_{s_1s_2}^\ast(\bm{u}_{s_1s_2});v_1)\\
  &\phantom{{}={}}\cdot\sum_{k=1}^{d_{s_1s_2\cdot}}(-v_1)^k\sum_{l=k}^{d_{s_1s_2\cdot}}\biggl(a_{s_1s_2,lk}(t_{s_1s_2}^\ast(\bm{u}_{s_1s_2}))\\
   & \phantom{{}={}}\cdot\biggl(\, \sum_{\bm{j} \in \mathcal{Q}^{d_{s_1s_2}}_{\bm{d}_{s_1s_2},l}}
    \prod_{s_3=1}^{d_{s_1s_2}} B_{d_{s_1s_2s_3},j_{s_3}}\bigl
    ((\cpsi_{s_1s_2,s_1s_2s_3}^{(k)}(t_{s_1s_2s_3}(\bm{u}_{s_1s_2s_3})))_{k\in\{ 1,\dots, d_{s_1s_2s_3}- j_{s_3}+1\} }\bigr)\biggr)\biggr)  \end{align*}
  which can be written as
  \begin{align}
  &\phantom{{}={}}\frac{\partial}{\partial\bm{u}_{s_1s_2}}\psi_{s_1,s_1s_2}(t_{s_1s_2}^\ast(\bm{u}_{s_1s_2});v_1)\notag\\
  &=\biggl(\,\prod_{s_3=1}^{d_{s_1s_2}}\prod_{s_4=1}^{d_{s_1s_2s_3}}(\psi_{s_1s_2s_3}^{-1})^\prime(u_{s_1s_2s_3s_4})\biggr)\psi_{s_1,s_2s_2}(t_{s_1s_2}^\ast(\bm{u}_{s_1s_2});v_1) \notag \\
  &\phantom{{}={}} \cdot \sum_{k=1}^{d_{s_1s_2\cdot}} a_{s_1s_2,d_{s_1s_2\cdot}k}(\bm{t}_{s_1s_2}(\bm{u}_{s_1s_2}))(-v_1)^k,\label{integrand.der}
  \end{align}
  for
  \begin{align}
    &\phantom{{}={}}a_{s_1s_2,d_{s_1s_2\cdot}k}(\bm{t}_{s_1s_2}(\bm{u}_{s_1s_2}))\notag\\
&=\sum_{l=k}^{d_{s_1s_2\cdot}}\biggl(a_{s_1s_2,lk}(t_{s_1s_2}^\ast(\bm{u}_{s_1s_2}))\notag\\
&\phantom{{}={}}\cdot\biggl(\,\sum_{\bm{j} \in \mathcal{Q}^{d_{s_1s_2}}_{\bm{d}_{s_1s_2},l}}
    \prod_{s_3=1}^{d_{s_1s_2}} B_{d_{s_1s_2s_3},j_{s_3}}\bigl
    ((\cpsi_{s_1s_2,s_1s_2s_3}^{(k)}(t_{s_1s_2s_3}(\bm{u}_{s_1s_2s_3})))_{k\in\{ 1,\dots, d_{s_1s_2s_3}- j_{s_3}+1\} }\bigr)\biggr)\biggr).\label{as.nk.2}
  \end{align}
  With this notation, the connection with Theorem \ref{main.theorem}
  \ref{main.theorem.1} is clearly visible.

  Finally, in order to solve Challenge 2 and 3, one introduces
  \begin{align}
    b^{d_{s_1}}_{\bm{d}_{s_1},k}(\bm{t}(\bm{u})) =  \sum_{\bm{j} \in \mathcal{Q}_{\bm{d}_{s_1},k}^{d_1}} \prod_{s_2=1}^{d_{s_1}}a_{s_1s_2,d_{s_1s_2\cdot}j_{s_2}}(\bm{t}_{s_1s_2}(\bm{u}_{s_1s_2})),\label{bk.2}
  \end{align}
  with
  \begin{align*}
    \bm{t}(\bm{u})&=(\bm{t}_1(\bm{u}_1)\T,\dots,\bm{t}_{d_1}(\bm{u}_{d_1})\T)\T,\\
    \bm{t}_{s_1}(\bm{u}_{s_1})&=(\bm{t}_{s_11}(\bm{u}_{s_11})\T,\dots,\bm{t}_{s_1d_{s_1}}(\bm{u}_{s_1d_{s_1}})\T)\T,\\
    \bm{d}_{s_1}&=(\bm{d}_{s_11}\T,\dots,\bm{d}_{s_1d_{s_1}}\T)\T.
  \end{align*}
  Similarly to Theorem \ref{main.theorem} \ref{main.theorem.2}, one then obtains (with
  $t(\bm{u})=\psiis{1}(C(\bm{u}))$ as before)
  \begin{align}
  c(\bm{u}) & = \int_0^\infty \prod_{s_2=1}^{d_{s_1}}  \frac{\partial}{\partial \bm{u}_{s_1s_2}} \, \psi_{s_1,s_1s_2}(t_{s_1s_2}^\ast(\bm{u}_{s_1s_2}); v_1) \, dF_1(v_1) \notag \\
  & =   \biggl(\,\prod_{s_2=1}^{d_{s_1}}  \prod_{s_3=1}^{d_{s_1s_2}}
  \prod_{s_4=1}^{d_{s_1s_2s_3}}( \psi_{s_1s_2s_3}^{-1})^\prime(u_{s_1s_2s_3s_4})\biggr)
  \sum_{k=d_{s_1}}^d b^{d_{s_1}}_{\bm{d}_{s_1},k}(\bm{t}(\bm{u}))\notag \\
  & \phantom{={}} \cdot \int_0^\infty \biggl(\,\prod_{s_2=1}^{d_{s_1}}
  \psi_{s_1,s_1s_2}(t_{s_1s_2}^{\ast}(\bm{u}_{s_1s_2});v_1) \biggr)  (-v_1)^k dF_1(v_1) \notag \\
  & = \biggl(\,\prod_{s_2=1}^{d_{s_1}}  \prod_{s_3=1}^{d_{s_1s_2}}
  \prod_{s_4=1}^{d_{s_1s_2s_3}}( \psi_{s_1s_2s_3}^{-1})^\prime(u_{s_1s_2s_3s_4})\biggr) \sum_{k=d_{s_1}}^d b^{d_{s_1}}_{\bm{d}_{s_1},k}(\bm{t}(\bm{u}))\,\psi_1^{(k)}(t(\bm{u})). \label{dnac3lvl}
\end{align}

\begin{remark}
  The pattern to compute the density of nested Archimedean copulas with more
  than three levels can be deduced from the previous computations, the following
  heuristic argument shows how. In order to understand the reasoning, it is
  useful to remind ourselves that the structure of nested Archimedean copulas
  can be depicted by trees; see Figure~\ref{fig.nAC3.tree} for a tree
  representation of (\ref{nAC3}). Let $L$ denote the number of levels (with
  (\ref{nAC3}) having $L=3$). Thanks to our notation, we can easily identify a
  certain branch of the tree with the corresponding sequence of indices. Each
  time a nesting level is added, for each branch $s_1s_2\cdots s_{l}$, $l \in
  \{1,\dots,L\}$, finite sequences of coefficients $a$'s and $b$'s (similar to
  $a_{s_1s_2, d_{s_1s_2\cdot}k}$ and $b^{d_{s_1}}_{\bm{d}_{s_1},k}$ above) will
  appear and their structure can be deduced from Equation (\ref{as.nk.2}).  More
  precisely, as in Equations (\ref{bk}) and (\ref{bk.2}), the sequence of $b$'s
  can always be interpreted as the coefficients of the Cauchy product of the
  polynomials with $a$'s as coefficients.

  The structure of the $a$'s at each branch $s_1\dots s_{l}$ is more
  complicated. For any branch $s_1 \dots s_{L}$, that is on the ultimate level
  of nesting, the $a$'s are simply the Bell polynomials applied to the function
  $\cpsi_{s_1\dots s_{L-1}, s_1\cdots s_{L}}$ and its derivatives. For any other
  branch $s_1 \dots s_{l}$, $l \in \{ 1,\dots, L-1 \}$, the coefficients $a$'s
  are the (Euclidean) inner product of the vector of all Bell polynomials
  applied to the function $\cpsi_{s_1\dots s_{l}, s_1\dots s_{l+1}}$ and its
  derivatives with the vector of all coefficients $b$'s, the exact structure of
  Equation (\ref{as.nk.2}). In Equation (\ref{as.nk.2}), the level $l$ is equal
  to $1=L-2$ and the term $a_{s_1s_2,lk}(t_{s_1s_2}^\ast(\bm{u}_{s_1s_2}))$
  stands for the Bell polynomial applied to $\cpsi_{s_1\dots s_{l}, s_1\dots
    s_{l+1}}$ and its derivatives, while the $l$th member of the sequence of
  $b$'s is defined by
\begin{align*}
\sum_{\bm{j} \in \mathcal{Q}^{d_{s_1s_2}}_{\bm{d}_{s_1s_2},l}}
    \prod_{s_3=1}^{d_{s_1s_2}}B_{d_{s_1s_2s_3},j_{s_3}}\biggl
    (\bigl(\cpsi_{s_1s_2,s_1s_2s_3}^{(k)}(t_{s_1s_2s_3}(\bm{u}_{s_1s_2s_3}))\bigr)_{k\in\{ 1,\dots, d_{s_1s_2s_3}- j_{s_3}+1 \} }\biggr),
\end{align*}
the term appearing in (\ref{as.nk.2}).
\end{remark}

\subsection*{Acknowledgements}
The authors would like to thank Paul Embrechts (ETH Zurich) for giving us
feedback and the opportunity to work on this topic.

\appendix
\section{Appendix}
\begin{proof}[Proof of Lemma \ref{lem.Bell}]
  \myskip \begin{enumerate}[label=(\arabic*),leftmargin=*,align=left,topsep=\mytopsep,itemsep=\myitemsep]
  \item $\sum_{l=1}^{n-k+1}(x-l)j_l=x\sum_{l=1}^{n-k+1}j_l-\sum_{l=1}^{n-k+1}lj_l=xk-n$.
  \item The identity $B_{n,k}(1,\dots,1)=S(n,k)$ can be found, for example, in \cite[p.\ 135]{comtet1974} or \cite[p.\ 87]{charalambides2005}. It then follows that
    \begin{align*}
      B_{n,k}(x,\dots,x)&=\sum_{\bm{j}\in\mathcal{P}_{n,k}}\binom{n}{j_1,\dots,j_{n-k+1}}\prod_{l=1}^{n-k+1}\biggl(\frac{x}{l!}\biggr)^{j_l}\\
      &=x^kB_{n,k}(1,\dots,1)=S(n,k)x^k,
    \end{align*}
    since $\sum_{l=1}^{n-k+1}j_l=k$ by definition of $\mathcal{P}_{n,k}$.
  \item By definition of $\mathcal{P}_{n,k}$ it follows from $\sum_{l=1}^{n-k+1}lj_l=n$ and \ref{lem.Bell.2} that
    \begin{align*}
      B_{n,k}(-x,\dots,(-1)^{n-k+1}x)&=\sum_{\bm{j}\in\mathcal{P}_{n,k}}\binom{n}{j_1,\dots,j_{n-k+1}}\prod_{l=1}^{n-k+1}\biggl(\frac{(-1)^lx}{l!}\biggr)^{j_l}\\
      &=(-1)^nB_{n,k}(x,\dots,x)=(-1)^nS(n,k)x^k.
    \end{align*}
  \item By (\ref{Bell}),
    \begin{align}
B_{n,k}(g^\prime(x),g^{\prime\prime}(x),\dots,g^{(n-k+1)}(x))=\sum_{\bm{j}\in\mathcal{P}_{n,k}}\binom{n}{j_1,\dots,j_{n-k+1}}\prod_{l=1}^{n-k+1}\biggl(\frac{g^{(l)}(x)}{l!}\biggr)^{j_l}.\label{der.prod}
    \end{align}
    Note that $\sign(g^{(l)}(x))=(-1)^{l-1}$ for all $x$ and $l\in\{1,\dots,n-k+1\}$, so that the sign of the $l$th factor in the product in (\ref{der.prod}) is $(-1)^{(l-1)j_l}$. This implies that
    \begin{align*}
      \sign\prod_{l=1}^{n-k+1}\biggl(\frac{g^{(l)}(x)}{l!}\biggr)^{j_l}=(-1)^{\sum\limits_{l=1}^{n-k+1}(l-1)j_l}=(-1)^{\sum\limits_{l=1}^{n-k+1}lj_l-\sum\limits_{l=1}^{n-k+1}j_l}.
    \end{align*}
    Now since we sum over $\bm{j}\in\mathcal{P}_{n,k}$, we see from (\ref{Bell.Pnk}) that
    \begin{align*}
      \sign\prod_{l=1}^{n-k+1}\biggl(\frac{g^{(l)}(x)}{l!}\biggr)^{j_l}=(-1)^{n-k}
    \end{align*}
    and thus the whole sum in (\ref{der.prod}) has this sign.
  \end{enumerate}
\end{proof}

\begin{proof}[Proof of Proposition \ref{Bell.id}]
  \myskip
  \begin{enumerate}[label=(\arabic*),leftmargin=*,align=left,topsep=\mytopsep,itemsep=\myitemsep]
  \item Let $h(t)=\exp((yt)^x)$. It follows from Fa\`a di Bruno's formula ((\ref{FdB}) and (\ref{FdB.Bell}) with $f(x)=\exp(x)$ and $g(t)=(yt)^x$) and Lemma \ref{lem.Bell} Part \ref{lem.Bell.1} that
    \begin{align*}
      h^{(n)}(t)&=h(t)\sum_{k=1}^n\sum_{\bm{j}\in\mathcal{P}_{n,k}}\binom{n}{j_1,\dots,j_{n-k+1}}\prod_{l=1}^{n-k+1}\biggl(\frac{(x)_ly^xt^{x-l}}{l!}\biggr)^{j_l}\\
      &=h(t)\sum_{k=1}^n\sum_{\bm{j}\in\mathcal{P}_{n,k}}\binom{n}{j_1,\dots,j_{n-k+1}}\prod_{l=1}^{n-k+1}\biggl(\frac{(x)_ly^x}{l!}\biggr)^{j_l}t^{xk-n}\\
      &=h(t)\sum_{k=1}^ny^n\sum_{\bm{j}\in\mathcal{P}_{n,k}}\binom{n}{j_1,\dots,j_{n-k+1}}\prod_{l=1}^{n-k+1}\biggl(\frac{(x)_ly^{x-l}}{l!}\biggr)^{j_l}t^{xk-n}\\
      &=\frac{h(t)}{t^n}\sum_{k=1}^ny^nB_{n,k}((x)_1y^{x-1},\dots,(x)_{n-k+1}y^{x-(n-k+1)})t^{xk}.
    \end{align*}
    On the other hand, we may differentiate the series expansion of $h$ and use
    the identity $e^{-x}\sum_{k=0}^\infty k^lx^k/k!=\sum_{k=0}^lS(l,k)x^k$, see
    \cite{boyadzhiev2009}, with $x$ being $(yt)^x$. With \ref{Stirling.first},
    it then follows that
    \begin{align*}
      h^{(n)}(t)&=\sum_{k=0}^\infty(xk)_n\frac{t^{xk-n}}{k!}y^{xk}=\frac{1}{t^n}\sum_{k=0}^\infty\biggl(\,\sum_{l=1}^n s(n,l)(xk)^l\biggr)\frac{(yt)^{xk}}{k!}\\
      &=\frac{1}{t^{n}}\sum_{l=1}^{n}s(n,l) x^l\sum_{k=0}^{\infty}k^l\frac{(yt)^{xk}}{k!}=\frac{\exp((yt)^x)}{t^n}\sum_{l=1}^{n}s(n,l)x^l\sum_{k=0}^lS(l,k)(yt)^{xk}\\
      &=\frac{h(t)}{t^n}\sum_{k=0}^n\biggl(y^{xk}\sum_{l=k}^nx^ls(n,l)S(l,k)\biggr)t^{xk}=\frac{h(t)}{t^n}\sum_{k=1}^n\biggl(y^{xk}\sum_{l=k}^nx^ls(n,l)S(l,k)\biggr)t^{xk}
    \end{align*}
    Comparing the two representations for $h^{(n)}$ leads to the result as stated.
  \item Since $\sum_{k=1}^j(x)_kS(j,k)=x^j$ and by (\ref{Stirling.first}), one obtains that
    \begin{align*}
      \sum_{k=1}^{n}(-1)_ks_{nk}(x)=\sum_{j=1}^ns(n,j)x^j\sum_{k=1}^j(-1)_kS(j,k)=\sum_{j=1}^ns(n,j)(-x)^j=(-x)_n.
    \end{align*}
  \item For all $x\in(0,1]$, $s_{nk}(x)=(-1)^{n-k}p(n;k)n!/k!$, where the probability mass function $p(n;k)>0$ (in $n\in\{k,k+1,\dots\}$) corresponds to the distribution function whose Laplace-Stieltjes transform is the inner generator appearing in a nested Joe copula; consider \cite[p.\ 99]{hofert2010c} with $V_0=k$ and $\theta_0/\theta_1=x$ to see this. This representation implies that $\sign(s_{nk}(x))=(-1)^{n-k}$.
  \end{enumerate}
\end{proof}

\begin{proof}[Proof of Theorem \ref{main.theorem}]
  \myskip
  \begin{enumerate}[label=(\arabic*),leftmargin=*,align=left,topsep=\mytopsep,itemsep=\myitemsep]
  \item Apply (\ref{FdB.Bell}) with $f(x)=\exp(-v_0x)$ and $g(t)=\cpsi_{0s}(t)$. For the statement about the signs, apply Lemma \ref{lem.Bell} \ref{lem.Bell.4}. For the last statement, note that by Lemma \ref{lem.Bell} \ref{lem.Bell.2}, $a_{s,11}(t)=B_{1,1}(\cpsi_{0s}^\prime(t))=B_{1,1}(1)=S(1,1)\cdot1^1=1$ if $\psi_s=\psi_0$.
  \item Given the form (\ref{psi0sder}) we see that the product appearing as integrand in (\ref{dnac}) can be computed via
    \begin{align*}
      \prod_{s=1}^{d_0}\psis{0s}{(d_s)}(t_s(\bm{u}_s);v_0)&=\prod_{s=1}^{d_0}\sum_{k=1}^{d_s}a_{s,d_sk}(t_s(\bm{u}_s))(-v_0)^k\cdot\prod_{s=1}^{d_0}\psi_{0s}(t_s(\bm{u}_s);v_0)\\
      &=\sum_{k=d_0}^d\var{b}{d_0}{\bm{d},k}(\bm{t}(\bm{u}))(-v_0)^k\cdot\prod_{s=1}^{d_0}\psi_{0s}(t_s(\bm{u}_s);v_0).
    \end{align*}
    Now note that $\prod_{s=1}^{d_0}\psi_{0s}(t_s(\bm{u}_s);v_0)=\exp(-v_0 t(\bm{u}))$. Hence, we obtain
    \begin{align*}
      \prod_{s=1}^{d_0}\psis{0s}{(d_s)}(t_s(\bm{u}_s);v_0)=\sum_{k=d_0}^d\var{b}{d_0}{\bm{d},k}(\bm{t}(\bm{u}))(-v_0)^k\exp(-v_0 t(\bm{u})).
    \end{align*}
    By replacing $v_0$ by $V_0$ and taking the expectation, one obtains
    \begin{align*}
      \IE\biggl[\,\prod_{s=1}^{d_0}\psis{0s}{(d_s)}(t_s(\bm{u}_s);V_0)\biggr]&=\sum_{k=d_0}^d\var{b}{d_0}{\bm{d},k}(\bm{t}(\bm{u}))\IE\bigl[(-V_0)^k\exp(-V_0t(\bm{u}))\bigr]\\
      &=\sum_{k=d_0}^d\var{b}{d_0}{\bm{d},k}(\bm{t}(\bm{u}))\psis{0}{(k)}(t(\bm{u})),
    \end{align*}
    so that, by (\ref{dnac}), the density is of the form as stated.
  \end{enumerate}
\end{proof}

\begin{proof}[Proof of Lemma \ref{FdbMultiStrong}]
Due to Equation (\ref{FdBMulti.hyp1}), observe that the second sum in (\ref{FdBMulti}) may be rewritten as
\begin{align*}
\sum_{\pi:\lvert\pi\rvert=k}\prod_{B\in\pi}\frac{\partial^{\lvert
      B\rvert}}{\prod_{i\in B}\partial x_i}g(\bm{x}) =
\sum_{\bm{j} \in \var{\mathcal{Q}}{m}{\bm{d},k}} \prod_{l=1}^m  \biggl(\, \sum_{\pi_l:\lvert \pi_l \rvert = j_l} \prod_{B \in\pi_l} \frac{\partial^{\lvert
      B\rvert}}{\prod_{i\in B}\partial x_i}g(\bm{x}) \biggr),
\end{align*}
where $\pi_l$ is a partition of $B_l$ with $j_l$ elements, $l\in\{1,\dots,m\}$; note that $j_l\leq d_l$, because $\lvert B_l \rvert = d_l$. Both sides are
equal because  the remaining terms in the sum on the left-hand side are zero as
derivatives with respect to two or more variables belonging to different
partitions vanish. Combining this result with (\ref{h1h2}), it follows from
(\ref{FdBMulti}) that
\begin{align*}
\frac{\partial}{\partial\bm{x}}f(g(\bm{x}))&=\sum_{k=1}^n f^{(k)}(g(\bm{x})) \sum_{\bm{j} \in \var{\mathcal{Q}}{m}{\bm{d},k}} \prod_{l=1}^m  \Biggl(\, \sum_{\pi_l:\lvert \pi_l \rvert = j_l}\biggl(\prod_{B \in\pi_l}\biggl(h_{l1}^{(\lvert B \rvert)}(h_{l2}(\bm{x})) \prod_{i \in B} \frac{\partial}{\partial x_i} h_{2l}(\bm{x})\biggr)\biggr)\Biggr).
\end{align*}
Observe that for any partition $\pi_l$ of $B_l$, $l\in \{ 1,\dots,m \}$, we have that
\begin{align*}
\prod_{B \in\pi_l} \prod_{i \in B} \frac{\partial}{\partial x_i} h_{l2}(\bm{x})
= \prod_{i \in B_l}  \frac{\partial}{\partial x_i} h_{l2}(\bm{x}),
\end{align*}
so that
\begin{align*}
\frac{\partial}{\partial\bm{x}}f(g(\bm{x}))&=\sum_{k=1}^n f^{(k)}(g(\bm{x}))
\sum_{\bm{j} \in \var{\mathcal{Q}}{m}{\bm{d},k}} \prod_{l=1}^m
\biggl(\biggl(\,\prod_{i \in B_l} \frac{\partial}{\partial x_i} h_{l2}(\bm{x})\biggr)
\sum_{\pi_l:\lvert \pi_l \rvert = j_l} \prod_{B \in\pi_l} h_{l1}^{(\lvert B
  \rvert)}(h_{l2}(\bm{x}))\biggr)\\
&=\biggl(\,\prod_{l=1}^m \prod_{i \in B_l} \frac{\partial}{\partial x_i} h_{l2}(\bm{x}) \biggr) \sum_{k=1}^n f^{(k)}(g(\bm{x})) \sum_{\bm{j} \in \var{\mathcal{Q}}{m}{\bm{d},k}} \prod_{l=1}^m \sum_{\pi_l:\lvert \pi_l \rvert = j_l} \prod_{B \in\pi_l} h_{l1}^{(\lvert B
  \rvert)}(h_{l2}(\bm{x})).
\end{align*}
Finally, applying Identity (\ref{BellCombin}) completes the proof.
\end{proof}

\printbibliography[heading=bibintoc]
\end{document}